\documentclass[a4paper,USenglish,draft]{scrartcl}
 
\newif{\ifArxiv}
\Arxivtrue

\usepackage{authblk}
\usepackage[USenglish]{babel}

\usepackage{amssymb,amsmath}
\usepackage{tabularx}

\usepackage{microtype}
\usepackage{tikz}
\usetikzlibrary{shapes}
\usetikzlibrary{arrows,decorations.markings}
\usepackage{pgf}
\usetikzlibrary{automata}
\usetikzlibrary{chains,decorations.pathmorphing,positioning,fit}
\usetikzlibrary{decorations.shapes,calc,backgrounds}
\usetikzlibrary{decorations.text,matrix}
\usetikzlibrary{arrows,shapes.geometric,shapes.symbols,scopes}
\usetikzlibrary{decorations.markings}
\usetikzlibrary{positioning,arrows}

\newcommand{\Card}[1]{\left|\mathinner{#1}\right|}
\usepackage[hidelinks, final]{hyperref}
\usetikzlibrary{arrows, automata,chains,fit,shapes}
\usetikzlibrary{decorations.markings}

\newcommand{\Oh}{\mathcal{O}}

\usepackage{amsfonts}
\usepackage{amssymb}
\usepackage{amsmath,xspace}
\usepackage[arrow, matrix, curve]{xy}
\usepackage{tikz}
\usetikzlibrary{shapes}
\usetikzlibrary{arrows,automata}

\usetikzlibrary{trees}
\usetikzlibrary{decorations}
\usetikzlibrary{snakes}
\usepackage{ellipsis}
\usepackage{ifthen}

\usepackage{amsthm}
\usepackage{prettyref}

\newcommand{\DSPACE}{\ensuremath{\mathsf{DSPACE}}\xspace} 
\newcommand{\PSPACE}{\ensuremath{\mathsf{PSPACE}}\xspace} 
 
\newcommand{\NSPACE}{\mathsf{NSPACE}}
\newcommand{\NTIME}{\mathsf{NTIME}}

\newcommand{\NP}{\ensuremath{\mathsf{NP}}\xspace} %

\renewcommand{\P}{\ensuremath{\mathsf{P}}\xspace}

\newcommand\ie{i.\,e., }
\newcommand\Wlog{W.\,l.\,o.\,g.\ }

\newcommand\eg{e.\,g.\xspace}

\newrefformat{thm}{Theorem~\ref{#1}}
\newrefformat{lem}{Lemma~\ref{#1}}
\newrefformat{dfn}{Definition~\ref{#1}}
\newrefformat{def}{Definition~\ref{#1}}
\newrefformat{cor}{Corollary~\ref{#1}} 
\newrefformat{prop}{Proposition~\ref{#1}}
\newrefformat{sec}{Section~\ref{#1}}
\newrefformat{kap}{Chapter~\ref{#1}}
\newrefformat{bsp}{Example~\ref{#1}}
\newrefformat{rem}{Remark~\ref{#1}}
\newrefformat{fig}{Figure~\ref{#1}}
\newrefformat{eq}{(\ref{#1})}
\newrefformat{ex}{Example~\ref{#1}}
\newrefformat{tab}{Table~\ref{#1}}

\newtheorem{theorem}{Theorem}

\newtheorem{lemma}[theorem]{Lemma}

\newtheorem{proposition}[theorem]{Proposition}
\newtheorem{corollary}[theorem]{Corollary}
\theoremstyle{definition}
\newtheorem{definition}[theorem]{Definition}
\newtheorem{remark}[theorem]{Remark}

\newcommand{\ov}[1]{\overline{#1}}

\newcommand{\vecdel}{\vec\delta}

\newcommand{\abs}[1]{\left|\mathinner{#1}\right|}
\newcommand{\Abs}[1]{\left\Vert\mathinner{#1}\right\Vert}

\newcommand{\set}[2]{\left\{\, \mathinner{#1}\vphantom{#2}\: \left|\: \vphantom{#1}\mathinner{#2} \right.\,\right\}}
\newcommand{\oneset}[1]{\left\{\, \mathinner{#1} \,\right\}}

\newcommand{\smallset}[1]{\left\{\mathinner{#1}\right\}}
\newcommand{\oi}[1]{{#1}^{-1}}

\newcommand{\sse}{\subseteq}
\newcommand{\es}{\emptyset}

\newcommand{\N}{\mathbb{N}}
\newcommand{\Z}{\mathbb{Z}}

\newcommand{\R}{\mathbb{R}}

\newcommand{\cG}{\mathcal{G}}

\newcommand{\cR}{\mathcal{R}}

\newcommand{\cvF}{\mathcal{V}}

\newcommand{\dist}{d}

\newcommand{\bigR}{2\cdot\Theta\cdot\Xi}

\renewcommand{\phi}{\varphi}

\newcommand{\WP}[1]{\mathop{\mathrm{WP}}({#1})}

\newcommand{\Sig}{\Sigma}

\newcommand{\GG}{\Gamma}

\newcommand{\alp}{\alpha}
\newcommand{\bet}{\beta}

\newcommand{\del}{\delta}

\newcommand{\bs}{\backslash}

\newcommand\Copt{\cC_{\mathrm{opt}}}

\newcommand\RA[1]{\underset{#1}{\Longrightarrow}}

\newcommand{\diam}{\mathrm{diam}}

\newcommand{\Comp}[1]{\overline{#1}}

\newcommand{\Stab}{\mathrm{Stab}}
\newcommand{\cC}{\mathcal{C}}

\newcommand{\ssnq}{\subsetneqq}

\newcommand{\sm}{\setminus}

\newcommand{\tto}{\overset{\sim}{\longrightarrow}}

\newcommand{\wt}[1]{\widetilde{#1}}

\newcommand{\Gram}{\mathbb{G}}

\newenvironment{aw}{\noindent\color{red} AW }{}
\newenvironment{awn}{\noindent\color{violet} AW new: }{}
\newenvironment{gs}{\noindent\color{blue} GS }{}

\renewcommand{\setminus}{\smallsetminus}

\renewcommand{\kappa}{k}

\title{The isomorphism problem for finite extensions of free groups is in \PSPACE}

\begin{document}

\author[1]{G\'eraud S\'enizergues}
\author[2]{Armin Wei\ss}
\affil[1]{
\normalsize	LABRI, Bordeaux, France\\ 
	\texttt{geraud.senizergues@u-bordeaux.fr}}
\affil[2]{ Universit{\"a}t Stuttgart, Germany\\
	\texttt{armin.weiss@fmi.uni-stuttgart.de}}


\maketitle

\vspace{-4mm}
\begin{abstract}
	\small
We present an algorithm for the following problem: given a context-free grammar for the word problem of a virtually free group $G$, compute 
a finite graph of groups ${\cal G}$ with finite vertex groups and fundamental group $G$.
Our algorithm is non-deterministic and runs in doubly exponential time.
It follows that	the isomorphism problem of context-free groups can be solved 
in doubly exponential space. 
%
Moreover, if, instead of a grammar, a finite extension of a free group is given as input, the construction of the graph of groups is in \NP and, consequently, the isomorphism problem in \PSPACE.

\vspace{-2mm}
 \end{abstract}

\vspace{-2mm}
\ifArxiv
\tableofcontents
\fi

\section{Introduction}

The study of algorithmic problems in group theory was initiated by
Dehn \cite{dehn11} when he introduced the word and the
isomorphism problem. The word problem asks whether some word over a (finite) set of
generators  represents the identity of the group. It also can be viewed
as a formal language, namely $\phi^{-1}(1) \sse \Sigma^*$ for some surjective homomorphism $\phi: \Sigma^* \to G$. The
isomorphism problem receives two finite presentations as input, the question is whether the groups they define are isomorphic. Although both
these problems are undecidable in general \cite{nov55,boone59}, there
are many classes of groups where at least the word problem can be
decided efficiently.

One of these classes are the finitely generated \emph{virtually free} groups (groups with a free
subgroup of finite index). It is easy to see that the word problem of a
finitely generated virtually free group can be solved in linear
time. Indeed, it forms a deterministic context-free
language. A seminal paper by Muller and Schupp \cite{ms83} shows the
converse: every group with a context-free word problem is virtually
free. Since then, also a wide range of other characterizations of
virtually free groups have emerged~-- for a survey we refer to
\cite{Antolin11,DiekertW17crm}.

The isomorphism problem of virtually free groups is also decidable as Krsti\'c showed in \cite{krstic89} (indeed, later Dahmani and Guirardel showed that the isomorphism problem for all hyperbolic groups is decidable \cite{DahmaniG11}). Here the input consists of two arbitrary finite presentations with the promise that both define virtually free groups. Unfortunately, the approach in \cite{krstic89} does not give any bound on the complexity. 
For the special case where the input is given as finite extension of free groups or as context-free grammars for the word problems, S\'enizergues \cite{sen93icalp,sen96dimacs} showed that the isomorphism problem is primitive recursive. 
\renewcommand{\theenumi}{(\Alph{enumi})}
\renewcommand{\labelenumi}{(\Alph{enumi})}
\vspace{-3mm}
\paragraph*{Contribution.} We improve the complexity for the isomorphism problem by showing: 
\begin{enumerate}
	\item Given a context-free grammar for the word problem of a context-free group $G$, a graph of groups for $G$ with finite vertex groups can be computed in $\NTIME(2^{2^{\Oh(n^2)}})$  (Theorem \ref{thm:compute_gog_cfg}).\label{result1}
	\item Given a virtually free presentation for $G$, a graph of groups for $G$ with finite vertex groups can be computed in $\NP$ (Theorem \ref{thm:compute_gog_NP_outer}).\label{result2} 
	\item The isomorphism problem for context-free groups given as grammars is in $\DSPACE(2^{2^{\Oh(n^2)}})$ (Theorem \ref{thm:iso_NEXP2_cf}).\label{result3}
	\item The isomorphism problem for virtually free groups given as virtually free presentations is in $\PSPACE$ (Theorem \ref{thm:iso_PSPACE_outer}).\label{result4}
\end{enumerate}
 Here, a \emph{virtually free presentation} for $G$ consists of a free group $F$ plus a set of representatives $S$ for $F\bs G$ together with relations describing pairwise multiplications of elements from $F$ and $S$. Typical examples of virtually free presentations are finite extensions of free groups.
For non-deterministic function problems we use the convention, that every accepting computation must yield a correct result; but the results of different accepting computations might differ\footnote{Thus, \ref{result2} means that the graph of groups can be computed in $\mathsf{NPMV}$ in the sense of \cite{Selman94}. More precisely, it can be rephrased as follows: the multi-valued function mapping a 
virtually free presentation for $G$ into a pair $({\cal G},\varphi)$, where 
${\cal G}$ is a graph of groups and $\varphi: \pi_1({\cal G}) \rightarrow G$ is an isomorphism of polynomial size, is everywhere defined and belongs to the class $\mathsf{FNP}$ as defined in \cite{pap94}.}.

\ifArxiv
The results \ref{result3} and \ref{result4} can seen be to follow from
\ref{result1} and \ref{result2} by using parts of Krsti\'c's algorithm
for the isomorphism problem. Here, we present another approach based on
so-called \emph{slide moves} on the graph of groups. Indeed, we conclude from Forester's work on deformation spaces
\cite{Forester02} that two graphs of groups with finite
vertex groups and isomorphic fundamental groups can be transformed into each other by a sequence of slide moves (\prettyref{prop:forester}).
\else
The results \ref{result3} and \ref{result4} can seen be to follow from
\ref{result1} and \ref{result2} rather easily. Indeed, we conclude from Forester's work on deformation spaces
\cite{Forester02} that two graphs of groups with finite
vertex groups and isomorphic fundamental groups can be transformed into each other by a sequence of slide moves (\prettyref{prop:forester}).
\fi

Our approach for proving \ref{result1} and \ref{result2} is as follows: in both cases the algorithm simply guesses a graph of groups together with a map and afterwards it verifies deterministically whether the map is indeed an isomorphism. The latter can be done using standard results from formal language theory. The difficult part is to show the existence of a ``small'' graph of groups and isomorphism (within the bounds of \ref{result1} and \ref{result2}).

For this, we introduce the \emph{structure tree} theory by Dicks and Dunwoody \cite{DicksD89} following a slightly different approach by Diekert and Wei\ss\ \cite{DiekertW13} based on the \emph{optimal cuts} of the Cayley graph  (\prettyref{sec:cuts}). The optimal cuts can be seen as the edge set of some tree on which the group $G$ acts. By Bass-Serre theory, this yields the graph of groups we are aiming for. Vertices in the graph of groups are defined in terms of equivalence classes of optimal cuts. The key in the proof is to bound the size of the equivalence classes.  
Using Muller and Schupp's \cite{ms83} notion of $k$-triangulability, Sénizergues \cite{sen96dimacs} proved bounds on the size of finite subgroups and on the number of edges in a reduced graph of groups for a context-free group, from which we derive our bounds. 

%

\vspace{-3mm}
\paragraph*{Outline.}
After fixing our notation, we recall basic facts from Bass-Serre
theory and the results from \cite{sen96dimacs} and give a short
review on structure trees based on \cite{DiekertW13}. 
\prettyref{sec:bounds}, develops bounds on the size of the
vertices (= equivalence classes of cuts) of the structure tree. After that, we introduce virtually free presentations formally
and we derive stronger bounds for this case in
\prettyref{sec:outer}. \prettyref{sec:main} completes the proofs of \ref{result1}
and \ref{result2}.  Finally, in \prettyref{sec:iso} we derive \ref{result3} and
\ref{result4} and we conclude with some open
questions.
%


\renewcommand{\theenumi}{(\roman{enumi})}
\renewcommand{\labelenumi}{(\roman{enumi})}
\vspace{-1mm}
\section{Preliminaries}
\label{sec:preliminaries}

\vspace{-1mm}
\ifArxiv
\paragraph*{Sets.} In order to distinguish it from quotient groups, we write $A\setminus B$ for the difference of sets $A$ and $B$. Moreover, the cardinality of a set $A$ is denoted by $\abs{A}$.
\vspace{-3mm}
\fi

\paragraph*{Complexity.} \ifArxiv
We use standard $\Oh$
-notation for functions from $\N$ to 
non-negative reals $\R^{\geq 0}$. 
\fi
We use the following convention for non-deterministic function problems: each accepting computation path must yield a correct answer~-- though different accepting paths can compute different correct answers. We use this convention to define the classes \NP (non-deterministic polynomial time) and $\NTIME(f(n))$ (non-deterministic time bounded by $f(n)$). Otherwise, we use standard complexity classes \P (deterministic polynomial time), \PSPACE (polynomial space) and $\DSPACE(f(n))$ (deterministic space bounded by $f(n)$) for both decision and function problems.

\vspace{-3mm}
\paragraph*{Words.} An \emph{alphabet} is a (finite) set $\Sig$; an element $a \in \Sig$ is called a  \emph{letter}. The set $\Sig^n$
forms the set of \emph{words} of length $n$. The length of $w\in \Sig^n$
is denoted by $\abs w$. The set of all words is denoted by $\Sig^*$. It is the free monoid over $\Sig$~-- its neutral element is the empty word $1$.
\ifArxiv If we can write $w = uxv$, then we call $u$ a \emph{prefix}, $x$ a \emph{factor} and $v$ a \emph{suffix} of $w$. \fi

\vspace{-3mm}
\paragraph*{Context-free grammars.} 
We use standard notation for context-free grammars:
a context-free grammar, is a tuple $\Gram =(V, \Sigma, P,S)$ with variables $V$, terminals $\Sigma$, productions rules $P \sse V \times (V \cup \Sigma)^*$, and a start symbol $S$. We denote its size by $\Abs{\Gram } = \abs{V} + \abs{\Sigma} + \sum_{S \to \alpha \in P} \abs{\alpha}$.
It is in \emph{Chomsky normal form} if all production are of the form $S \to 1$, $A \to a$ or $A \to BC$ with $A,B,C \in V$, $a \in \Sigma$. For further definitions on context-free grammars
, we refer to \cite{HU}.

\vspace{-3mm}
\paragraph*{Groups.} 
%
We consider groups $G$ together with a finite subset of monoid generators $\Sigma$.
Every word $w \in \Sig^*$ is simultaneously viewed as the corresponding group element in $G$ under the 
canonical projection $\pi:\Sig^* \to G$. Whenever it is not clear whether equality is as group elements or words, we write $w=_{G}w'$ as a shorthand of 
$\pi(w)=\pi(w')$. \ifArxiv Thus, $w=_{G}w'$ means that $w$ and $w'$ represent the same element in the group $G$. \fi
The \emph{word problem} of $G$ is the formal language $\WP{G} = \pi^{-1}(1)$.

 A \emph{symmetric} set of generators is a set with the involution $a \mapsto \ov a =_G \oi a$ (\ie  $\ov{\ov a} = a$). 
Let $w \in \Sig^*$ and $\Sig$ be symmetric.
We say that $w$ is \emph{freely reduced} if there is no factor $a\ov a$ for any letter $a \in \Sig$.
Given an arbitrary set of generators $X$, the free group over $X$ is denoted by $F(X)$. It is defined as $(X \cup \ov X)^*$ modulo the defining relations $x \ov x = 1$ for $x \in X \cup \ov X$.

\vspace{-3mm}
\paragraph*{Graphs.}\label{sub:graphs}

A  \emph{(undirected) graph} $\Gamma = (V,E,s,t, \ov{\,\cdot\,})$ is given by the following data: 
A set of vertices $V= V(\Gamma)$, a set of edges $E=E(\GG)$ together with two \emph{incidence}
maps $s:E \to V$ and $t:E \to V$ and an involution $E \to E$, $e \mapsto \ov e$
without fixed points such that $s(e) = t(\ov e)$. \ifArxiv The vertex $s(e)$ is the \emph{source} of $e$ and $t(e)$ is the \emph{target} of $e$. \fi
The \emph{degree} of $u$ is the number of incident edges. 
\ifArxiv A graph is finite, if it has finitely many vertices and edges. \fi 
An \emph{undirected edge} is the set $\smallset{e, \ov e}$. For the cardinality of sets of edges we usually count the number of undirected edges.
\ifArxiv A \emph{directed graph} is a graph without the involution. \fi

A (finite) \emph{path} from $v_0$ to $v_n$ is a pair of sequences $((v_0,\ldots,v_n), (e_1,\ldots,e_n))$ such that $s(e_i) =v_{i-1}$ and $t(e_i)=v_{i}$ for all $1 \leq i \leq n$.  Similarly, a \emph{bi-infinite path} is a pair of sequences $((v_i)_{i \in \Z}, (e_i)_{i \in \Z})$ such that $s(e_i) =v_{i-1}$ and $t(e_i)=v_{i}$ for all $i \in \Z$. A path is \emph{simple} if 
the vertices are pairwise distinct. It is \emph{closed} if $v_0=v_n$.
Depending on the situation we also denote paths simply by the sequence of edges or the sequence of vertices. \ifArxiv Given two paths $\beta$ and $\gamma$, we denote the concatenation by $\beta\gamma$. \fi
The \emph{distance} $d(u,v)$ between vertices $u$ and $v$ is defined as the length (i.\,e., the number of edges) of a shortest path connecting $u$ and $v$. \ifArxiv We let 
$d(u,v) = \infty$ if there is no such path. \fi
\ifArxiv A path $v_0,\dots,v_n$ is called \emph{geodesic} if $n= d(v_0,v_n)$. \fi 
For $A,B\sse V(\Gamma)$ the distance is defined as $d(A,B) = \min\set{d(u,v)}{u\in A,v\in B}$. 
\ifArxiv An undirected graph $\Gamma$ is called \emph{connected} if $d(u,v) < \infty$ for all vertices $u$ and $v$. 
A tree is a connected graph which does not contain any non-trivial simple closed path. \fi

For $S\sse V(\GG)$ we define $\GG - S$ to be the induced subgraph
of $\GG$ with vertices $V(\GG) \setminus S$. For $C \sse V(\Gamma)$, we write $\Comp C$ for
the complement of $C$, i.\,e., $\Comp C = V(\GG) \sm C$.  We call $C$
connected, if the induced subgraph is connected.
A group $G$ acts on a graph $\Gamma$, if it acts on both $V(\Gamma)$ and $E(\Gamma)$ and the actions preserve the incidences.

\vspace{-3mm}
\paragraph*{Cayley graphs.} Let $G$ be a group and $\Sigma$ a symmetric 
generating set of $G$. (If $\Sigma$ is not symmetric, we simply add a set of formal inverses $\ov \Sigma$.)
 The \emph{Cayley graph} $\Gamma= \Gamma_\Sig(G)$ of $G$ (with respect to $\Sigma$) is defined by $V(\Gamma) = G$ and 
$E(\Gamma) = G \times \Sigma$, with the incidence functions 
$s(g,a) = g$, $t(g,a) = ga$, and involution $\ov{(g,a)} = (ga,a^{-1})$. 
\ifArxiv The Cayley graph is connected because $\Sigma$ generates $G$. 
 The \emph{directed Cayley graph} is defined analogously without requiring that $\Sigma$ is symmetric. \fi
 \ifArxiv\else
  For $r \in \N$ let $B(r):=\set{u\in V(\Gamma)}{\dist(u,1)\leq r}$ denote the ball with 
radius $r$ around the identity. \fi

\vspace{-3mm}
\paragraph*{Cuts.} \ifArxiv For $v\in V(\Gamma)$ let $B_v(r):=\set{u\in V(\Gamma)}{\dist(u,v)\leq r}$ denote the ball with 
radius $r$ and center $v$. If $\Gamma$ is a Cayley graph,  we also write $B(r)$ for the 
ball with radius $r$ and center $1$. \fi
For a subset $C\sse V(\GG)$ we define the \emph{edge} and \emph{vertex boundaries} of $C$ as follows: 
%
%
\begin{align*}
\vecdel C & =\set{e\in E(\Gamma)}{s(e)\in C, t(e)\in \Comp{C} }\text{ = directed edge boundary},\\
\ifArxiv
\delta C & =\set{e\in E(\Gamma)}{s(e)\in C, t(e)\in \Comp{C} \text{ or } t(e)\in C, s(e)\in \Comp{C}}\text{ = edge boundary},\\ 
\fi
\partial C&= \bigl\{\, s(e) \,\,\big\vert\,\,  e\in \vecdel C \,\bigr\}
\text{ = inner vertex boundary},\\
\beta C&= \bigl\{\, s(e) \,\,\big\vert\,\,  e\in \vecdel C \text{ or } \ov e\in \vecdel C \,\bigr\} = \partial C \cup \partial \ov C
\text{ = vertex boundary}.
\end{align*}

\begin{definition}\label{def:cut}
	A \emph{cut} is a subset $C \subseteq V(\GG)$  such that $C$ and $\Comp{C}$ are both non-empty and connected and \ifArxiv$\delta C$\else$\vecdel C$ \fi is finite.
	The \emph{weight} of a cut is $\vert \vecdel C\vert$\ifArxiv (so the weight of a cut is the number of undirected edges in $\delta C$)\fi. If $\vert\vecdel C\vert\leq k$, we call  $C$ a \emph{$k$-cut}. 
w\end{definition}

\vspace{-2mm}
\subsection{Bass-Serre theory}\label{sec:bass_serre}\label{sec:GX}
We give a brief summary of the basic definitions and results of Bass-Serre theory \cite{serre80}. 

\begin{definition}[Graph of Groups]
	Let $Y = (V(Y),E(Y))$ be a connected graph. 
	A \emph{graph of groups $\cG$ over $Y$} is given by the following data:
	\begin{enumerate}
		\item For each vertex $P \in V(Y)$ there is a \emph{vertex group} 
		$G_P$.
		\item For each edge $y \in E(Y) $ there is an \emph{edge group}
		$G_y $ such that $G_y = G_{\ov y}$.
		\item For each edge $y \in E(Y) $ there is
		an injective homomorphism from $G_y$ to $G_{s(y)}$, which is denoted by $a \mapsto a^y$.
		The image of $G_y$ in $G_{s(y)}$ is denoted by $G_y^y$.
	\end{enumerate}
	Since we have $G_y = G_{\ov y}$, there is also a homomorphism $G_y \to G_{t(y)}$
	with $a \mapsto a^{\ov y}$. The image of $G_y$ in $G_{t(y)}$ is denoted by $G_y^{\ov y}$.
	 A graph of groups is called \emph{reduced} if $G_y^y \neq G_{s(y)}$ whenever $s(y) \neq t(y)$ for $y \in E(Y)$. Throughout we assume that all graphs of groups are connected and finite (\ie $Y$ is a connected, finite graph). 
\end{definition}

\vspace{-5mm}
\paragraph*{Fundamental group of a graph of groups.}
We begin with the group $F(\cG)$. It is defined as the free product of the free group 
$F({E(Y)})$ and the 
groups $G_P$ for $P \in V(Y)$ modulo the set of defining relations
$\set{\ov{y}a^yy=a^{\ov{y}}}{a \in G_y, \, y \in E(Y)}$. 
As a set of (monoid) generators we fix the disjoint union 
$\Delta= \biguplus_{P\in V(Y)} (G_P \setminus \smallset{1}) \cup E(Y)$
throughout. Now, we have
\[F(\cG) = F(\Delta)/\!\set{ gh=[gh],\, \ov{y}a^yy=a^{\ov{y}}}{ P\in V(Y),\, g,h \in G_P;\, y\in E(Y), \, a\in G_y }\!,\]
where $[gh]$ denotes the element obtained by multiplying $g $ and $h$ in $G_P$.

%
%

%
For $P\in V(Y) $ we define a subgroup $\pi_1(\cG, P)$ of $F(\cG)$ by 
the elements 
$g_0y_1\cdots g_{n-1}y_ng_{n}\in F(\cG)$, such that
	$y_1\cdots y_n$ is a closed path from $P$ to $P$ and
	 $g_i\in G_{s(y_{i+1})}$  for $0\leq i < n$ and $ g_{n} \in G_{P}$.
The group $\pi_1(\cG,P)	$ is called the \emph{fundamental group} of $\cG$ with respect to the base point $P$.
%
%
%
%
%
%
Since we assumed $Y$ to be connected, there exists a 
spanning tree $T= (V(Y),E(T))$ of $Y$. The \emph{fundamental group} of $\cG $ with respect to $T$ is defined as 
\[\pi_1(\cG,T)= F(\cG)/\set{y=1}{y\in T}.\]
%

\begin{proposition}[\cite{serre80}]\label{prop:twofunds}
	The canonical homomorphism $\psi$ from the subgroup $\pi_1(\cG, P)$ 
	of $F(\cG)$ to the quotient group $\pi_1(\cG,T)$ is an isomorphism. In particular, the two definitions of the fundamental group are independent of the choice of the base point or the spanning tree.
\end{proposition}

A word $w \in \Delta^*$ is called \emph{reduced} if it does not contain a factor $gh$ with $g,h \in G_P$ for some $P$ or a factor $\ov{y}a^yy$ with $y\in E(Y)$, $a\in G_y$. 
\begin{lemma}[\,{\!\cite[Thm.\ I.11]{serre80}}]\label{lem:reduced_word}
A reduced word in $\pi_1(\cG,P)$ represents the trivial element if and only if it is the empty word.
\end{lemma}

\vspace{-3mm}
\paragraph*{The quotient of a $G$-tree.}\label{sec:GX}
Graphs of groups arise in a natural way in situations where a group $G$ acts (from the left) on some connected tree $X = (V,E)$ without edge inversion, i.\,e., $\ov e \notin Ge$ for all $e \in E$. 
We let $Y = G\bs X$ be the quotient graph with 
vertex set $V(Y) = \set{Gv}{v \in V}$ and edge set $E(Y) = \set{Ge}{e \in E}$ and incidences and involution induced by $X$. By choosing representatives we find embeddings $\iota: V(Y) \hookrightarrow V$ and $\iota: E(Y) \hookrightarrow E$ and we can assume that $\iota(V(Y))$ forms a connected subgraph of $X$ and that $\iota(\ov y) = \ov{\iota(y)}$.
For $P \in V(Y)$, $y\in E(Y)$, we define vertex and edge groups as the stabilizers of the respective representatives: $G_P= \Stab(\iota P ) = \set{g \in G}{g\iota P = \iota P}$ and $G_y= \Stab(\iota y ) = \set{g \in G}{g\iota y = \iota y}$. Note that as abstract groups the vertex and edge groups are independent of the choice of representatives since stabilizers in the same orbit are conjugate.
Moreover, for each $y \in E(Y)$, there are $P,Q \in V(Y)$ and $g_y,h_y \in G$ such that 
$s(\iota y) = g_y\iota P$ and $t(\iota y) = h_y\iota Q$. Note that $P$ and $Q$ are uniquely determined by $y$, whereas for $g_y$ and $h_y$ only the left cosets $g_yG_P$ resp.\ $h_yG_Q$ are uniquely determined. Hence, here is another choice involved; still we can choose them such that $g_y = h_{\ov y}$ and $h_{y} = g_{\ov y}$. This yields two embeddings: 
\begin{align}\label{eq:inclhoms}
G_y \to G_P &,\qquad a \mapsto a^y = \ov g_y a g_y, & &\text{ and }&
G_y \to  G_Q  &, \qquad a \mapsto a^{\ov y}= \ov h_y a h_y.
\end{align} 	
Hence, we have obtained a well-defined graph of groups over $Y$. 
Notice that the $G_y^y$ and $G_y^{\ov y}$ depend on the choice of $g_y$ and $h_y$ (and change via conjugation when changing them).

We define a homomorphism $\phi:\Delta^* \to G$ by  $\phi(g) = g$ for 
$g \in G_P$, $P \in V(Y)$. For  $y \in E(Y)$, we set
$\phi(y) = \ov g_y h_y$. That means $\phi(y)$ maps some edge in the preimage of $y$ and terminating in $\iota t(y)$ to an edge in the preimage of $y$ with source in $\iota s(y)$. By our assumption, we have $\ov{ \phi(y)} = \ov h_yg_y = \phi(\ov y)$. 
Since 
$ \phi(\ov y a^y y) =\phi(\ov y)\phi( a^y) \phi(y) = \ov h_y g_y a^y \ov g_y h_y = a^{\ov y} = \phi(a^{\ov y})$, we obtain a well-defined 
homomorphism $\phi: F(\cG) \to G$.
\begin{theorem}[\cite{serre80}]\label{thm:bst}
	\label{prop:gXX}
	The restriction $\phi: \pi_1(\cG,P)\to G$ is an isomorphism.
\end{theorem}

%
%
%

\vspace{-2mm}
\subsection{Context-free groups and graphs}
\label{sub:cf_groups}\label{sec:cf_groups}
In this section we recall some results from \cite{ms83,ms85,sen93icalp,sen96dimacs}.
\begin{definition}
	A group is called \emph{context-free}, if its word problem 
	 is a context free language.
\end{definition}
Notice that the word problem of a context-free group is decidable in polynomial time~-- even if the grammar is part of the input~- by applying the CYK algorithm (see \eg\ \cite{HU}).

\begin{definition}[k-triangulable]\label{def:k-triangulable}
	Let $\Gamma$ be a graph. Let $k\in\N$ and let $\gamma = v_0, v_1,\dots, v_n=v_0$ be a sequence of vertices $\Gamma$ such that $d(v_{i-1}, v_i) \leq k$ for all $i\in \oneset{1, \dots, n}$  (\eg $\gamma$ can be a closed path).
	Let $P$ a convex polygon in the plane whose vertices are labeled by the vertices of $\gamma$ (i.e. we consider $\gamma$ as a simple closed curve in the plane). 
	A \emph{$k$-triangulation} of $\gamma$ is a triangulation of $P$ which does not introduce any additional vertices (thus only consists of ``diagonal'' edges) and such that vertices joined by a diagonal edge are are distance at most $k$. \ifArxiv We call the diagonal edges \emph{chords}. A path witnessing that the end vertices of the chord are at distance at most $k$ is called a \emph{label} of the chord. \fi
	If $n < 3$, we consider $\gamma$ as triangulated.
	
	If every closed path $\gamma$ has a $k$-triangulation, then $\Gamma$ is called $k$-triangulable.
\end{definition}

\ifArxiv

\begin{lemma}[\,{\!\cite[Thm.\ I]{ms83}}]\label{lem:triangulation_k}
	Let $(V, \Sigma, P,S)$ be a context-free grammar in Chomsky normal form for the word problem 
	of $G$ where $\Sigma$ is a symmetric generating set. 
	Then the Cayley graph $\Gamma$ can be 
	$k$-triangulated for $k = 2^{\abs{P}}$. Moreover, if $\Sigma$ is not symmetric $\Gamma$ can be 
	$k$-triangulated for $k = 2^{\abs{P} + 2}$
\end{lemma}

\begin{proof}
	Let $(V, \Sigma, P,S)$ be a context-free grammar in Chomsky normal form for the word problem of $G$ and let $\Sigma$ be symmetric. In \cite{ms83} Muller and Schupp proved that $\Gamma$ is $k$-triangulable for some $k$. An easy induction shows that indeed $k = 2^{\abs{P}}$ suffices.
	
	In the case that $\Sigma \neq_G \Sigma ^{-1}$, the same holds but only for the directed Cayley-graph $\Gamma'$ (at all nodes we have an outgoing directed edge for each letter in $\Sigma$), i.e. $\Gamma'$ can be $k'$-triangulated for some $k' = 2^{\abs{P}}$. Now, for every $a\in \Sigma$ we can take a shortest word $w_a$ representing $a^ {-1}$ in $G$. By $k'$-triangulability of $\Gamma'$ we have that $w_a$ has length at most $3k'$ (otherwise we could find a shortcut). Now, if we have a closed path $\gamma$ in the undirected Cayley-graph $\Gamma$, every edge of $\gamma$ corresponds to a path of length at most $3k'$ in $\Gamma'$. The resulting path $\gamma'$ in $\Gamma'$ can be $k'$-triangulated. Since every vertex on $\gamma'$ is at distance at most $\frac{3k}{2}$, this triangulation gives rise to a $\frac{5k}{2}k'$-triangulation of $\gamma$. Hence, $\Gamma$ is $4k'$-triangulable.
\end{proof}

\begin{remark}
	From now on we always assume that $\Sigma$ is symmetric. As the previous lemma shows this gives only a linear blow-up to the triangulation constant $k$.
\end{remark}

\else

\begin{lemma}[\,{\!\cite[Thm.\ I]{ms83}}]\label{lem:triangulation_k}
	Let $(V, \Sigma, P,S)$ be a context-free grammar in Chomsky normal form for the word problem 
	of $G$ where $\Sigma$ is a symmetric generating set. 
	Then the Cayley graph $\Gamma$ can be 
	$k$-triangulated for $k = 2^{\abs{P}}$.
\end{lemma}
Note that in \cite{ms83} only the existence of some $k$ is shown; however, an easy induction shows the bound of \prettyref{lem:triangulation_k}. Moreover, the condition that $\Sigma$ is a symmetric generating set is not really necessary as shown in \cite{SenizerguesW18arxiv}.

\fi

\begin{lemma}[\,{\!\cite[p.65]{ms85}}]\label{lem:p65_ms85}
	Let $\GG$ be $\kappa$-triangulable and let $r\in \N$.  If $C$ is a connected component of $\Gamma - B(r)$, then             
	$\diam(\partial C)\leq 3\kappa$. 
\end{lemma}

\begin{lemma}\label{lem:lemma6_icalp93}
	Let $\GG$ be connected and $k$-triangulable and let $C\sse V(\GG)$  be a cut. 
Then 
	$\diam(\beta C) \leq \frac{3k}{2} \vert\vecdel C\vert$.  
\end{lemma}
This lemma is asserted (without proof) in \cite[Lemma 6]{sen93icalp} with a slightly worse bound on $\diam(\beta C)$.

\begin{proof}
	
	We show that for every $v_0 \in \beta C$ we have \[\beta C \sse B_{v_0}\left(\frac{3k}{2}  \left(\vert\vecdel C\vert-1\right) + 1\right).\]  More precisely, we will inductively construct an enumeration $\vecdel C = \oneset{e_0,  \dots, e_n}$ for $n= \vert\vecdel C\vert - 1$ such that $v_0$ is an endpoint of $e_0$ and $\dist(v_0, t(e_i)),\dist(v_0, s(e_i)) \leq \frac{i\cdot 3k}{2} +1$ for $i > 0$. 

	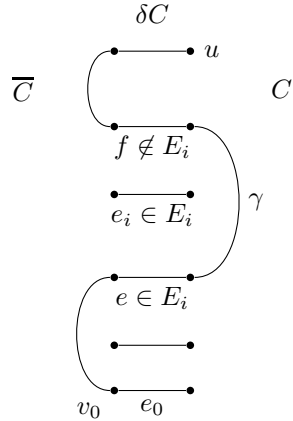
\begin{figure}[thb]
	\begin{center}
		\small
		\begin{tikzpicture}[dots/.style={fill, circle, inner sep = 1pt}, arrow/.style={bend angle=90}]
		\node[dots, label=below left:$v_0$] (vo) at(0,0){};
		\node[dots] (vo0) at (1,0){};
		
		\node[dots] (w0) at(0,0.6){};
		\node[dots] (w00) at (1,.6){};
		
		\node[dots] (w1) at(0,1.5){};
		\node[dots] (w10) at (1,1.5){};
		\node[dots] (w2) at(0,2.6){};
		\node[dots] (w20) at (1,2.6){};
		\node[dots] (w3) at(0,3.5){};
		\node[dots] (w30) at (1,3.5){};
		\node[dots] (u) at(0,4.5){};
		
		\node[dots, label=right:$u$] (u0) at (1,4.5){};
		
		\draw (vo) --node[below] {$e_0$} (vo0);
		\draw (w0) -- (w00);
		\draw (w1) --node[below] {$e\in E_i$} (w10);
		\draw (w2) --node[below] {$e_i \in E_i$} (w20);
		\draw (w3) --node[below] {$f \not\in E_i$} (w30);
		\draw (u) -- (u0);
		
		\node[] (x) at(0.5,5){$\delta C$};
		
		\draw[arrow, bend left] (vo) to (w1);
		\draw[arrow, bend right] (w10) to node[right]{$\gamma$}  (w30);
		\draw[arrow, bend left] (w3) to (u);
		
		\node at (-1.2,4){$\Comp C$};
		\node at (2.2,4){$C$};
		
		\end{tikzpicture}

	\end{center}
	\caption[]{A shortest path $\gamma$ from $v_0$ to $u$. We have $e\in E_i = \oneset{e_0,\ov e_0 \dots, e_i, \ov e_i}$ and $f \in \delta C \setminus E_i$ }\label{fig:lemma6_icalp93a}
\end{figure}

	Let $e_0$ be an edge with endpoint $v_0$. Then both endpoints of $e_0$ have distance at most one to $v_0$.
	
	Let $ \oneset{e_0, \dots, e_i}$ be constructed for some $i < n$. %
	Now let $u \in \beta C$ be a vertex which is not an endpoint of an edge in $ \oneset{e_0, \dots, e_i}$ and take a shortest path $\gamma$ from $v_0$ to $u$ in $\Gamma$. We denote $E_i = \oneset{e_0,\ov e_0 \dots, e_i, \ov e_i}$. Let $e$ be the last edge on $\gamma$ from $E_i $ and let $f$ the next edge on $\gamma$ belonging to $\delta C $ (\ie\  $f \in \delta C \setminus E_i$)~-- see \prettyref{fig:lemma6_icalp93a}.

	Let $\gamma'$ be the subpath of $\gamma$ connecting $t(e)$ to $s(f)$~-- \ie $\gamma'$ is a shortest path from $t(e)$ to $s(f)$. \Wlog we may assume that $\gamma'$ is contained in $C$. Since also $\Comp{C}$ is connected, there is also a path $\gamma''$ from $t(f)$ to $s(e)$ inside $\Comp{C}$. 
	
	The closed path $\gamma'f\gamma''e$ can be $k$-triangulated. If $\gamma'$ and $\gamma''$ are of length zero, we take $e_{i+1} = f$ and we are done by induction. Otherwise, there must be at least one triangle having two vertices in $\gamma'$ and one in $\gamma''$ or vice-versa. Moreover, all the triangles of this form build a linearly ordered chain where on one end there is one triangle with a side labeled by $e$ and on the other end a triangle labeled by $f$. Now, all these triangles have two sides which correspond to paths crossing $\delta C$~-- hence these paths all contain an edge of $\delta C$. Now consider the last triangle such that one side (and only one side) contains an edge from $ E_i$~-- lets say $e_j$~-- and one edge from $\delta C \setminus E_i$. The latter edge (in its correct orientation) is the new edge $e_{i+1}$. Now, the length of the three sides of the triangle together is at most $3k$. Therefore, the distance between any two vertices in any of the sides is at most $3k/2$. In particular, \[\max \set{ \dist(u,v) }{u \in \smallset{s(e_j), t(e_j)},\, v \in \smallset{s(e_{i+1}), t(e_{i+1})}}  \leq \frac{3k}{2}\] and we are done by induction.
	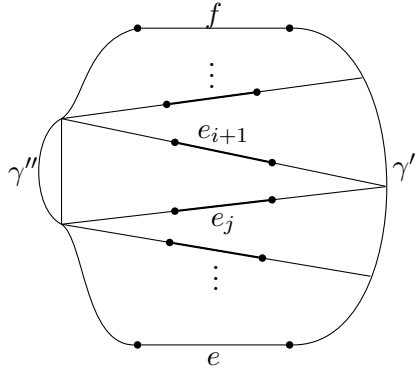
\begin{figure}
		\begin{center}		
			\begin{tikzpicture}[every node/.style={inner sep = 0pt,outer sep = 0pt}, dots/.style={fill, circle, inner sep = 1pt}, arrow/.style={bend angle=90}]
			
			\node[dots] (w1) at(-1,0){};
			\node[dots] (w10) at (1,0){};
			
			\node[dots] (w3) at(-1,4.2){};
			\node[dots] (w30) at (1,4.2){};
			
			\draw (w1) --node[below= 3pt] {$e$} (w10);
			
			\draw (w3) --node[above] {$f$} (w30);
			\draw[arrow, bend right] (w10) to node[pos=.3] (ll) {} node (lmid) {} node[above right= 3pt ] {$\gamma'$} node[pos=.75] (lu) {} (w30);
			
			\draw (w1) ..controls (-1.7,0) and (-1.7,1.4) .. (-2,1.6) ..controls (-2.4,1.8) and (-2.4,2.8) .. node[left] {$\gamma''$} (-2,3)  ..controls(-1.7,3.2) and (-1.7,4) .. (w3);
			
			\draw (-2,1.6) -- (-2,3) ;
			\draw  (-2,1.6) -- node[pos=.35,dots] (al) {} node[pos=.65,dots] (bl) {} (lmid) ;
			\draw[thick] (al) -- node[below= 3pt ] {$e_j$} (bl);
			
			\draw  (lmid) to node[pos=.35,dots] (au) {} node[pos=.65,dots] (bu) {} (-2,3) ;
			
			\draw  (lu) to node[pos=.35,dots] (auu) {} node[pos=.65,dots] (buu) {} (-2,3) ;
			\draw  (-2,1.6) -- node[pos=.35,dots] (all) {} node[pos=.65,dots] (bll) {} (ll) ;
			\draw[thick] (all) -- node[below= 0pt ] (all) {$\vdots$}  (bll);
			\draw[thick] (auu) -- node[above= 4pt ] (all) {$\vdots$}  (buu);

			\draw[thick] (au) -- node[above= 3pt ] {$e_{i+1}$} (bu);
			
			
			
			\end{tikzpicture}
		\end{center}
		\caption[]{There is one triangle with one side containing an edge $ e_j \in \oneset{e_0,\ov e_0 \dots, e_i, \ov e_i}$ and one side containing an edge from $\delta C \setminus \oneset{e_0,\ov e_0 \dots, e_i, \ov e_i}$. This second edge is the new edge $e_{i+1}$. }\label{fig:lemma6_icalp93b}
	\end{figure}
\end{proof}

The following upper-bounds will be useful.

\begin{proposition}[\,{\!\cite[Prop.\ 1.2]{sen96dimacs}}]\label{prop:subgroupbound}
	Let $\Gamma$ be the Cayley graph on $X \cup \ov{X}$ of a group $G$ and let us suppose that 
	$\Gamma$ is $k$-triangulable.  Let $H\leq G$ be a finite subgroup. Then 
	$\abs{H} \leq (2\cdot\abs{X})^{12k + 10}.$
\end{proposition}

\begin{theorem}[\,{\!\cite[Thm.\ 1.4]{sen96dimacs}}]\label{thm:1.4_dimacs96}
	Let $\Gamma$ be the Cayley graph on $X$ of a group $G$ and let us suppose 
that $\Gamma$ is
	$k$-triangulable. 
	Then every minimal graph of groups ${\cal G}$ admitting $G$ as fundamental
	group has at most
	${(2 \cdot \abs{ X })}^{12 k+11}$ undirected edges.
\end{theorem}

\vspace{-2mm}
\subsection{Optimal cuts and structure trees}\label{sec:cuts}

We briefly present the construction of optimal cuts and the associated structure tree from \cite{DiekertW13,DiekertW17crm}. While in
	\cite{DiekertW13}, the proof was for arbitrary accessible co-compact and locally
	finite graphs, here we assume that $\GG$ is the Cayley graph of a context-free group.
%
%
%
%
%
We are interested in bi-infinite simple paths which can be split into two infinite pieces by some cut. 
For a bi-infinite simple path $\alp$ denote:
\begin{align*}
\cC(\alp) &= \set{C \sse V(\GG)}{\text{$C$ is a cut and } \abs{\alp \cap C} = \infty = \abs{\alp \cap \Comp C}},\\
\cC_{\min}(\alp) &= \set{C \in \cC(\alp)}{\text{$\abs{\del C}$ is minimal in }\cC(\alp)},
\end{align*}
where we identify $\alpha$ with its set of vertices. 
If $\cC(\alp)\neq \es$, we say that $C$ \emph{splits} $\alpha$. 
We define the set of \emph{minimal cuts} $\cC_{\min}$ as the union over the  $\cC_{\min}(\alp)$ for all all bi-infinite simple paths $\alpha$.
%
%
%
\ifArxiv
Note that  $\cC_{\min}$ may contain cuts of very different weight. 
Actually we might have $C,D \in \cC(\alp) \cap \cC_{\min}$
with $C \in \cC_{\min}(\alp)$, but  $D \notin \cC_{\min}(\alp)$. 
In such a case there must be another  bi-infinite simple path $\bet$ 
with  $D \in \cC(\alp) \cap \cC_{\min}(\bet)$ and $\abs{\del C} < \abs{\del D}$.

For example, let $\GG$ be the  subgraph of the infinite grid
$\Z \times \Z$ which is induced by the pairs $(i,j)$ satisfying  $j\in \{0,1\}$ or $i=0$ and $j\geq 0$. 
Let $\alp$ be the bi-infinite simple path with $i=0$ or $j=1$ and $i\geq 0 $ and let $\bet$ be the bi-infinite simple path defined by $j=0$. 
Then there are such cuts with $\abs{\del C} = 1$ and $\abs{\del D} = 2$, as depicted in \prettyref{fig:different_deltas}.


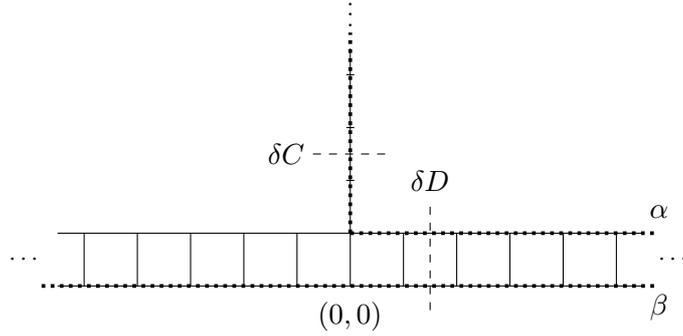
\begin{figure}[ht]
	\begin{center}
		\begin{tikzpicture}[scale = 0.7]
		\def\width{5};
		\def\height{4};
		
		\draw (-\width -0.5,1) -- (\width +0.5,1) ;
		\draw (-\width -0.5,0) -- (\width +0.5,0) ;
		\draw (0,0) -- (0,\height + 0.5) ;
		\node () at (-\width -1.1, 0.5) {\footnotesize{$\cdots$}};
		\node () at (+\width +1.1, 0.5) {\footnotesize{$\cdots$}};
		\node () at (0, \height + 1.3) {\footnotesize{$\vdots$}};
		\node (oo) at (0,0){};
		\node [below=2pt] {$(0,0)$};
		\foreach \x in {1,...,\width}
		{
			\draw (\x,1) -- (\x,0);
			\draw (-\x,1) -- (-\x,0);
		}
		\foreach \y in {2,...,\height}
		{
			\draw (-0.075,\y) -- (0.075,\y);
		}
		
		\node () at (1.5, 2.0) {$\delta D$};
		\draw[dashed] (1.5,1.5) -- (1.5,-0.5) ;
		\node () at (-1.2, 2.5) {$\delta C$};
		\draw[dashed] (-0.7,2.5) -- (0.7,2.5) ;

		\draw[dotted,line width=1.4pt] (-\width -0.8,0) -- (\width +0.8,0) ;

		\draw[dotted,line width=1.4pt] (0,1) -- (0,\height + 0.8) ;
		\draw[dotted,line width=1.4pt] ( 0,1) -- (\width +0.8,1) ;
		\node () at (\width +0.8, 1.4) {$\alpha$};
		\node () at (\width +0.8, -0.4) {$\beta$};
		
		\end{tikzpicture}
	\end{center}
	\caption[]{
		The subgraph of the grid $\Z \times \Z$ induced by the pairs $(i,j)$ satisfying  $j\in \{0,1\}$ or $i=0$ and $j\geq 0$. Here we have $D \in \cC(\alp) \cap \cC_{\min}$ but $D \notin \cC_{\min}(\alp)$.}\label{fig:different_deltas}
\end{figure}
\fi
\ifArxiv
Let us prove some bounds on weight and diameter of minimal cuts:

\begin{lemma}\label{lem:diameter_mincuts}
	Let $\Gamma$ be $k$-triangulable and let $d$ denote the degree of $\Gamma$. Then for every $C \in \cC_{\mathrm{min}}$ we have 
	\begin{enumerate}
		\item $\vert \vecdel C\vert \leq d^{3k+2}$ and 
		\item $\diam (\beta C)\leq \frac{3k}{2} d^{3k+2}$. 
	\end{enumerate}
\end{lemma}

\begin{proof}
	Let $C$ be a connected component of $\Gamma - B(m)$ for any $m \in \N$. By 
	\prettyref{lem:p65_ms85} we have $\diam(\partial C)\leq 3k$. Hence, $\vert\vecdel C\vert \leq  \sum_{i=1}^{\diam(\partial C)} d^i \leq d^{3k+2}$.
	Since any bi-infinite path which can be split by some cut into two infinite paths also can be 
	split into two infinite part by a cut which is a connected component of $\Gamma - B(m)$ 
	for some $m \in \N$ (just make $m$ large enough that the boundary of the  cut splitting the path lies entirely within $B(m)$), we know that $\vert\vecdel C\vert \leq d^{3k+2}$ also for 
	every minimal cut $C$.
	By Lemma \ref{lem:lemma6_icalp93} it follows that 
	$\diam (\beta C)  \leq \frac{3k}{2} \vert\vecdel C\vert
	\leq \frac{3k}{2}d^{3k+2}$
	for every minimal cut $C$.
\end{proof}

\else
Since every bi-infinite path can be split by some cut which is a connected component of $\Gamma - B(m)$ for some $m \in \N$, by \prettyref{lem:p65_ms85} and \ref{lem:lemma6_icalp93}, we easily obtain  the following bounds. 

\begin{lemma}\label{lem:diameter_mincuts}
	Let $\Gamma$ be $k$-triangulable and let $d$ denote the degree of $\Gamma$. Then for every $C \in \cC_{\mathrm{min}}$ we have 
$\vert \vecdel C\vert \leq d^{3k+2}$ and
$\diam (\beta C)\leq \frac{3k}{2} d^{3k+2}$. 
\end{lemma}
\fi
%
%
%
	Two cuts $C$ and $D$ are called \emph{nested}, if one of the four inclusions
	$C\sse D$, $C\sse \Comp{D}$, $\Comp{C}\sse D$ or $\Comp{C}\sse\Comp{D}$ holds.
%
By \prettyref{lem:diameter_mincuts}, with $K=d^{3k + 3}$ for every 
bi-infinite simple path $\alp$ with $ \cC(\alp)\neq \es $ there exists some  cut $ C \in \cC(\alp) $ 
with $\vert \vecdel C\vert \leq K$. 
We fix this number $K$. 
For a cut $ C $ let $m(C)$ denote the number of $K$-cuts that are not nested with $C$. It follows from \cite{ThomassenW93} that $m(C)$ is always finite, see also \cite[Lem.\ 3.4]{DiekertW13}.
This allows us to define the set of \emph{optimal cuts}:
\begin{align*}
\Copt(\alp)&=  \set{C\in \cC_{\min}(\alp)} {m(C)\leq m(D) \text{ for all } D \in \cC_{\min}(\alp)},\\
\Copt&= \bigcup\set{\Copt(\alp)}{\alp \text{ is a  bi-infinite simple path} }.
\end{align*}

\begin{definition}
	A set $\cC\sse \cC(\Gamma)$ of cuts is called a \emph{tree set}, if $\cC$ is pairwise nested, closed under complementation and for each $C,D\in\cC$ the set $\set{E \in \cC}{C\sse E\sse D}$ is finite.
\end{definition}

\begin{proposition}[\cite{DiekertW13}]\label{prop:opt_nested}
	$\Copt$ is a tree set.
\end{proposition}

\begin{definition}\label{def:cuts_equiv}
	Let $\cC$ be a tree set. We can now define the following relation:
\[	C\sim D :\Longleftrightarrow C=D \text{ or } (\Comp{C} \ssnq  D \text{ and }
	E\in \cC , \,\Comp{C} \ssnq  E \sse D\RA{} E=D).\]
\end{definition}
Indeed, $\sim$ is an equivalence relation~-- see \eg\ \cite{DicksD89}.
The intuition behind this definition is: We consider $\cC$ as edge set of a graph, and define two edges to be incident to the same vertex, if no other edge lies ``between'' them.

\begin{definition}
	Let $\cC$ be a tree set and let  $T(\cC)$ denote the graph defined by
	\begin{align*}
	V(T(\cC))&=\set{[C]}{C \in \cC},&
	E(T(\cC))&= \cC.
	\end{align*}
	The incidence maps are defined by $s(C)=[C]$ and $t(C)=[\Comp{C}]$. The involution $\Comp{C}$ is defined by the complementation $\Comp{C} = V(\GG)\sm C$; hence, we do not need to change notation.
\end{definition}
The directed edges are in canonical bijection with the pairs $([C],[\Comp{C}])$. Indeed, let  $C\sim D$ and $\Comp{C}\sim \Comp{D}$. It follows $C=D$ because otherwise $C\ssnq \Comp{D}\ssnq C$.
Thus, $T(\cC)$ is an undirected graph without self-loops and multi-edges. Indeed, $T(\cC)$ is a tree \cite{DicksD89}. 


%
%

\begin{theorem}[\,{\!\cite[Thm.\ 5.9]{DiekertW13}}]\label{thm:DiekertW13}\label{thm:finite_vertex_stabilizers}
	Let $\Gamma$ be a connected, $k$-triangulable, locally finite graph. 
	Let a group $G$ act on $\Gamma$ such that $G\bs \Gamma$ is finite and each 
	node stabilizer $G_v$ is finite. 
	Then $G$ acts on the tree $T(\Copt)$ such that all vertex and edge 
	stabilizers are finite and $G\bs T(\Copt)$ is finite.
\end{theorem}


\section{Bounds on the structure tree}\label{sec:bounds}

In order to prove our main result, we have to show that there exists a ``small'' graph of groups together with a ``small'' isomorphism.  For constructing such a graph of groups, we start with the structure tree and bound the size of the equivalence classes and the diameter of the boundaries of the cuts in one equivalence class. 
As before $\Gamma$ is the Cayley graph of a context-free group $G$.

\vspace{-3mm}
\paragraph*{Complete cut sets.}
By \prettyref{prop:opt_nested} and \prettyref{thm:finite_vertex_stabilizers}, $\Copt$ is a tree set of cuts on which $G$ acts with finitely many orbits such that the vertex stabilizers $G_{[C]} = \set{g \in G}{gC\sim C}$ of the structure tree are finite. We call a set of cuts with these properties a \emph{complete cut set}.

\vspace{-3mm}
\paragraph*{Avoiding edge inversion.} We aim to construct a graph of groups as described in \prettyref{sec:GX} from the structure tree $T(\Copt)$.
However, if the action of $G$ on $T(\Copt)$ is with edge inversion, the construction cannot be applied directly. Instead, we switch to a subdivision $\wt T(\cC)$ of $T(\cC)$ by
putting a new vertex in the middle of every edge which is inverted (in particular, $V(T(\cC)) \sse V(\wt T(\cC))$). 
Formally, $\wt T(\cC)$ is defined as follows: for every edge $C$ of $T(\Copt)$ with $gC = \Comp C$ we remove $C$ and $\ov C$ and instead add a new vertex $\smash{ v_{\{C,\ov C\}}}$ together with edges $C_1,\ov C_1,C_2,\ov C_2$ with $gC_1 = \ov C_2$, $gC_2 = \ov C_1$ and $s(C_1) = [C]$, $t(C_1) = v_{\{C,\ov C\}}$, $s(C_2) = v_{\{C,\ov C\}}$, and $t(C_2) = [\ov C]$.

%
%

\vspace{-3mm}
\paragraph*{Reduced cut sets.} Given a complete cut set $\cC$, we obtain a graph of groups with finite vertex groups by taking the quotient $G\bs \wt T(\cC)$ with the procedure from \prettyref{sec:GX}. We aim to apply \prettyref{thm:1.4_dimacs96}, to bound the number of edges in this graph of groups. However, the graph of groups might not be reduced. In terms of the structure tree $T(\cC)$ this means that there are either vertices $[C]$ and $[\Comp{C}]$ which are not in the
same $G$-orbit and $G_{[C]} = G_C $ or that $G_{[C]} = G_C $ and there is some $g \in G$ with $gC = \Comp C$. Nevertheless, in this
case we can switch to a subset $\wt\cC\sse \cC$ such that still all vertex stabilizers are finite and the corresponding graph of groups is reduced: if there is some cut $C \in \cC$ with $G_{[C]} = G_C $ and either $[C] \not\in G\cdot \smallset{[\Comp{C}]}$ or  $\Comp C \in G \cdot\smallset{C}$, then we can replace $\cC$ by $\cC \setminus G\cdot\smallset{C,\Comp{C}}$ (in terms of the structure tree this means we collapse the respective edges). If there is no such $C \in \cC$ anymore, we have obtained a \emph{reduced} set of cuts $\cC$.
	The following lemmas are straightforward to verify.
	\begin{lemma}\label{lem:reduced_complete}
		Let $\cC$ be a complete cut set and let $\cC'$ be the reduced cut set obtained by the above procedure. Then $\cC'$ is also complete (\ie all vertex stabilizers are still finite).
\end{lemma}

\begin{proof}
	We only need to proof that still all vertex stabilizers (or equivalently all equivalence classes~-- since $G$ acts freely on $\Gamma$) are finite. Consider the structure tree  $T(\cC)$. If we remove some cut from  $\cC$ it means we collapse the respective edge. 
	
	Let $C \in \cC$ with $[C] \not\in G\cdot \smallset{[\Comp{C}]}$ and $G_{[C]} = G_C $ and consider the set of cuts $\cC \setminus G\cdot \smallset{C, \ov C}$. Thus, $[C] \cap G \cdot \smallset{C} = \smallset{C}$. Let $P$ the vertex in the structure tree with $\ov C \in P$. Removing $G \cdot\smallset{C, \ov C}$ amounts to replacing $P = [\ov C]$ by 
	\begin{align*}
	P' = \left([\Comp{C}] \cup \bigcup_{g\in G_{P}}g[C]\right) \setminus \bigcup_{g\in G_{P}}\oneset{gC, g \Comp C},
	\end{align*}
	which is also finite.
	
	Now, let $C \in \cC$ with $\ov C \in G\cdot \smallset{C}$ and $G_{[C]} = G_C $ and consider the cut set $\cC \setminus G\cdot \smallset{C} = \cC \setminus G\cdot \smallset{C, \ov C}$. Then we have $[C] \cap G \cdot \smallset{C} = \smallset{C}$ and $[\ov C] \cap G \cdot \smallset{\ov C} = \smallset{\ov C}$. Let $P$ the vertex in the structure tree with $\ov C \in P$. Removing $G \cdot\smallset{C}$ amounts to replacing $P = [C]$ by 
	$	P' = \left([C] \cup [\ov C]\right) \setminus\smallset{C, \ov C}$,
	which again is finite. 
\end{proof}

	\begin{lemma}\label{lem:reduced_edge_inversion}
		Let $\cC$ be a reduced cut set and let $\wt T(\cC)$ be the associated structure tree without edge inversion. Then the graph of group built on $G\bs \wt T(\cC)$ is reduced.
\end{lemma}

\begin{proof}
	Let $\cG$ be a graph of groups over $Y$ and let $\iota: V(Y) \to V (\wt T(\cC))$ and $\iota: E(Y) \to E (\wt T(\cC))$ be the choice of representative as in \prettyref{sec:GX}.
	Assume that $\cG$ is not reduced, \ie there is an edge $y \in E(Y)$ with $G_y^{y}= G_{s(y)}$ and $t(y) \neq s(y)$. Then $\iota s(y) \in V(T(\cC))$ (\ie $s(y)$ is not in the orbit of one of the additional vertices introduced on inverted edges) because otherwise the index of $G_y^y$ in $G_{s(y)}$ would be two. We distinguish the two cases that $\iota t(y)$ is one of these special vertices and one of the original vertices of $T(\cC)$.
	If $\iota t(y)$ is in $T(\cC)$, then $\iota y = C$ for some $C \in \cC$ and $[C] \not \in G\cdot \smallset{[\ov C]}$ and $G_C = G_{[C]}$, \ie $\cC$ is not reduced.
	
	On the other hand, let $\iota t(y) = v_{\{C,\ov C\}}$ for some $C \in \cC$. Then by the construction of $\wt T(\cC)$, we have $\ov C \in G \cdot \smallset{C}$. Since by assumption $G_C = G_{[C]}$, that means again that $\cC$ is not reduced.
\end{proof}

Let $\Xi$ be an upper bound on the order of finite subgroups of $G$ and let $\Theta$ be a bound on the number of undirected edges in a reduced graph of groups for $G$. Notice that by  \prettyref{prop:subgroupbound}, we have $\Xi \leq d^{12k + 10}$ and by  \prettyref{thm:1.4_dimacs96} we have $\Theta \leq d^{12k + 11}$ where $k$ is the triangulation constant and $d$ the degree of $\Gamma$. \ifArxiv\else The following lemma is straightforward to prove using the fact that every orbit of cuts in $\cC$ yields an edge in the graph of groups. \fi
\begin{lemma}\label{lem:subgroupbound} 
	Let $\cC$ be a reduced complete set of cuts and let $C \sim D\in \cC$. Then
\ifArxiv	\begin{enumerate}
		\item $\abs{\set{g\in G}{gD \sim C}}\leq \Xi $, and 
		\item $\abs{[C]}\leq \bigR$.\label{classbound} 
	\end{enumerate}
\else
we have $\abs{\set{g\in G}{gD \sim C}}\leq \Xi $ and 
 $\abs{[C]}\leq \bigR$.
\fi
\end{lemma}


\begin{proof}
	Let $g \in G$ with $gD \sim C$ and let $E\in [C]$. Then, $gE \sim gC \sim gD \sim C$. Thus, $gE \in[C]$ and so  $g \in G_{[C]}$. By \prettyref{thm:finite_vertex_stabilizers}, $G_{[C]}$ is finite; thus, its size is bounded by $\Xi$. This shows the first point.
	
	For the second bound observe that, by the definition of $\Theta$ and by \prettyref{lem:reduced_edge_inversion}, 	
	there are at most $2\Theta$ different $G$-orbits of (directed) edges of $\wt T(\cC)$ (the 2 comes from the fact $\Theta$ counts undirected edges). Recall that  $\wt T(\cC)$ originated from $T(\cC)$ by adding new vertices in the middle of edges $C$ with $gC=\ov C$ for some $g \in G$. Thus, adding new vertices on these edges does not change the number of $G$-orbits of edges. Hence, $\abs{G \bs \cC} \leq 2 \Theta$.
	Now, by the first point, for every $G$-orbit there are at most $\Xi$ cuts in $[C]$. 
\end{proof}
	
	\begin{lemma}\label{lem:classcomplete}
		Let $\cC$ be a tree set of cuts and let $G$ act on $\cC$. Let $C \in \cC$ and  $C \in P\sse [C]$. 
		
		Then $P \neq [C]$  if, and only if,
		there is some $ E \in  [C] \setminus P $ with $d(\partial E, \bigcup_{D\in P} \partial D) \leq 1$. 
	\end{lemma}
	
	\begin{proof}

		The if-part is clear. Thus, let $P \neq [C]$. 
		Then there is some $E \in [C] \setminus P$. Since $ \Comp{E} \ssnq D$ for all $D \in P$, we have  $ \emptyset \neq \Comp{E} \sse \bigcap_{D\in P} D$. Now, if $\partial \Comp{E} \sse \bigcup_{D\in P}\partial D $, we are done. Otherwise, there is some vertex  $u \in\Comp{E} \sse  \bigcap_{D\in P} D $ with $\dist(u,\bigcup_{D\in P}\partial D)\geq 1$. 
\ifArxiv 
		Take the connected component $F$ of $\bigcap_{D\in P}  D $ such that $u \in F$.  Since $\partial F \sse \bigcup_{D\in P}\partial D$ (if $w\in \partial F$, then $w$ is adjacent to some vertex in $\Comp{D}$ for some $D \in P$), we find a vertex  $v \in F \sse \bigcap_{D\in P} D $ with $\dist(v,\bigcup_{D\in P}\partial D)= 1$.
\else
		 Since $\partial \bigl(\bigcap_{D\in P}  D \bigr) \sse \bigcup_{D\in P}\partial D$, we also find a vertex  $v \in \bigcap_{D\in P} D $ with $\dist(v,\bigcup_{D\in P}\partial D)= 1$ by following a path from $u$ to $\partial \bigl(\bigcap_{D\in P}  D \bigr)$ inside $\bigcap_{D\in P}  D$.
\fi
		Notice that, in particular, we have 
		\begin{align}
		v &\not\in \beta D \cup \Comp{D} \qquad\text{for all } D \in P.\label{eq:v_not_in_D}
		\end{align}

		Now since $\Gamma$ is vertex-transitive, we can find some cut $\wt E \in \cC$ such that $v \in \beta \wt E$.
		After possibly exchanging $\wt E$ with its complement, we can assume that $\wt E \ssnq C$ or $\wt E \sse \Comp{C}$. The latter would imply $v \in \beta \wt E \sse \beta C \cup \Comp{C}$ contradicting \prettyref{eq:v_not_in_D}.
		Moreover, for any other $D \in P$, we have  $\wt E \sse D$ because all other possibilities for $\wt E$ and $D$ being nested lead to a contradiction:
		\begin{itemize}
			\item if $\Comp D \ssnq \wt E $, then  $\Comp D\ssnq \wt E \ssnq C$ contradicting $D \sim C$
			\item if $D \ssnq \wt E $, then  $D \sse \wt E \sse C$ and $\Comp D \sse C$ contradicting $\Comp C \neq \emptyset$
			\item if $ \wt E \sse \Comp D$, then $v \in \beta \wt E \sse \beta D \cup \Comp{D}$ contradicting \prettyref{eq:v_not_in_D}.
		\end{itemize}
		Thus, $ \wt E \sse \bigcap_{D\in P} D$.
		Let $E\in \cC$ be minimal with respect to inclusion such that $\wt E \sse \Comp E \ssnq C$. Then $E \sim C$, but $E \not\in P$ because $v\in \beta \wt E \sse \beta E \cup \Comp{E}$.
		
		It remains to verify that $\dist(\partial E, \bigcup_{D\in P} \partial D) \leq 1$. Let $w \in \partial D $ for some $D\in P$ a vertex with $\dist(w,v) =1$. Then, we have $w \in \beta D \cup \Comp{D} \sse \beta E \cup E$. Consider the two cases: $v \in \Comp E$ and $v \in \partial E$. 
		If $v \in \Comp E$, then $w \in \beta E \cap \partial D $ and hence $d(\partial E, \bigcup_{D\in P} \partial D) \leq 1$. If $v \in \partial E$, then  $\dist(\partial E, \bigcup_{D\in P} \partial D) \leq \dist(v,w) = 1$.
	\end{proof}
\ifArxiv \else	
	Now, an easy inductive argument shows the next lemma.
\fi
	\begin{lemma}\label{lem:classdiameterbound} Let $\cC$ be a complete set of cuts and $R \in \N$ such that $\diam \beta C \leq R$ for all $C \in \cC$. Let $C \in \cC$, then 
\ifArxiv
	\[\diam \left( \bigcup_{C\sim D} \beta C \right) \leq (R+1) \cdot\abs{[C]} .\] 
\else
		$\displaystyle\diam \Bigl( \bigcup\limits_{C\sim D} \beta C \Bigr) \leq (R+1) \cdot\abs{[C]}.$ 
\fi
		
	\end{lemma}

	\begin{proof}
	We can construct the class $[C]$ by starting with the singleton set $\oneset{C}$ and successively adding cuts $D$ with $D\sim C$. More precisely, we construct a sequence of cuts $C_1, \dots, C_n$ such that $C_1 = C$, $C_i \sim C$ for all $i$ and $\beta C_i \sse B_v((R+1)i)$ for any fixed $v \in \beta C$. This will show the lemma.  
	
	By assumption, we have $\beta C_1 \sse B_v(R)$. Now, let $C_1, \dots, C_i$ be constructed. If $\oneset{C_1, \dots, C_i} \neq [C]$, by \prettyref{lem:classcomplete}, we find a cut $C_{i+1} \in [C]$ with  $\dist(\bigcup_{j=1}^i \partial C_j, \partial C_{i+1}) \leq 1$. Thus, by induction $\dist(v, \partial C_{i+1}) \leq 1 + (R+1)i$ and so   $\beta C_{i+1} \sse B_v((R+1)(i+1))$.
\end{proof}

	\begin{lemma}\label{lem:gog_bound}
		Let $\Gamma$ be the Cayley graph of $G$. Moreover, assume that
		\begin{itemize}
			\item  $R$ is an upper bound on the diameter of the boundary of minimal cuts,
			\item  $\Theta$ is an upper bound on the number of undirected edges of a reduced graph of group for $G$,
			\item $\Xi$ is an upper bound on the size of finite subgroups of $G$.
		\end{itemize}
%
		Then there exists a graph of groups ${\cal G}$ over $Y$ and an isomorphism
		$\varphi: \pi_1(\cG,T) \to  G$
		with
		\begin{enumerate}
			\item $\Card{V(Y)} \leq \Theta + 1$,\label{V_bound}
			\item $\Card{G_P} \leq \Xi $ for all $P \in V(Y)$,\label{G_P_bound}
			\item  $|\phi(a)| \leq 4(R+1)\cdot(\Theta + 1)^2 \cdot \Xi$ for every $a \in \bigcup_{P\in V(Y)} G_P\cup E(Y)$.\label{edge_bound}
		\end{enumerate}
	\end{lemma}

\begin{proof}
	First consider the case that $\Gamma$ is finite. In this case we
	have $V(Y) = \oneset{P}$ with $G \cong G_P$. Thus,
	$\abs{V(Y)} = 1 \leq \Theta + 1$ and $\Card{G_P} \leq \Xi$ by
	definition of $\Xi$. Moreover, the Cayley graph of $G$ has diameter
	at most $\Xi -1$ (as it has only $\Xi$ vertices). Therefore, every
	every $g\in G_P$ has an image of length at most
	$\Xi - 1 \leq 2(R+1)\cdot(\Theta + 1)^2 \cdot \Xi$.
	
	Now let $\Gamma$ be infinite. We start with the structure tree
	$T(\Copt)$.  By \prettyref{lem:reduced_complete}, we
	can switch to a complete subset $\cC \sse \Copt$ such that the
	corresponding graph of groups is reduced. Let $Y = G \bs \wt T(\cC)$ (as described at the beginning of \prettyref{sec:bounds}, $\wt T(\cC)$ is the tree obtained from $T(\cC)$ by putting additional vertices on inverted edges).  We need to choose proper
	representatives $\iota V(Y)$ and $\iota E(Y)$ to construct
	the graph of groups as in \prettyref{sec:GX}.
	
	Fix any $P \in V(Y)$ and choose a representative $\iota P \in V(\wt T(\cC)) $ such that $1 \in \bigcup_{C \in \iota P} \partial C$. Now, choose the other representatives according to \prettyref{sec:GX} such that $\iota V(Y)$ forms a connected subgraph and $\iota E(Y)$ is closed under complementation. This defines the vertex and edge groups as the respective stabilizers. For vertices $Q$ with  $\iota Q = v_{\{C,\ov C\}}$ (\ie vertices placed on inverted edges), the vertex group is the stabilizer $G_{v_{\{C,\ov C\}}} = \set{g \in G}{gC = C \text{ or } gC = \ov C}$. For the edge $y$ with $s(y) = Q$, we set $G_y = G_C$ if $C \in \iota P$ for some $P\in  V(Y) \setminus \smallset{Q}$ (\ie $C$ is the cut which connects $v_{\{C,\ov C\}}$ to the rest of $\iota V(Y)$).

	Now, $y\in E(Y)$ is incident to $P \in V(Y)$ if, and only if,  $ s(\iota y) = g_y \iota P$ for some $g_y\in G$ (resp.\ $ t(\iota y) = g_{\ov y} \iota P$). For every $y \in E(Y)$ we can fix such an element $g_y$ and 		
	define $\phi : F(\cG) \to G$ by 
	\begin{align*}
	&&&&	\phi(g) &= g &&\text{for } g \in G_P, P \in V(Y),&&
	&&&	\phi(y) &= g_y^{-1} g_{\ov y} && \text{for }y \in E(Y).&&
	\end{align*}
	By \prettyref{thm:bst} this induces an isomorphism $\pi_1(\cG,T) \to  G$.
	
	Points \ref{V_bound} and \ref{G_P_bound} hold by the very definition of $\Theta$ and $\Xi$ (since $Y$ is connected, it has at least $\abs{V(Y)} - 1$ edges).
	For point \ref{edge_bound}, let us first show that 
	\begin{align}
	\bigcup_{P \in V(Y)} \bigcup_{D \in P} \partial D \sse B(2(R+1)\cdot (\Theta + 1)\cdot\Theta \cdot \Xi + \Theta).\label{eq:tranvsversalbound}
	\end{align}
	
	For every equivalence class $P \in V(Y)$, by \prettyref{lem:classdiameterbound} and \prettyref{lem:subgroupbound}, we have $ \diam\left(\bigcup_{D \in P} \partial D\right) \leq (R+1)\cdot\bigR$. Moreover, by \ref{V_bound}, we have $\abs{ V(Y)} \leq \Theta + 1$. Since $\dist\left(\bigcup_{D \in P} \partial D, \bigcup_{D \in Q} \partial D\right) \leq 1$ if there is an edge (\ie  a cut) connecting $P$ to $Q$, we obtain  \prettyref{eq:tranvsversalbound}. 
	Let us write $\Lambda$ for $2(R+1)\cdot (\Theta + 1)\cdot\Theta \cdot \Xi + \Theta $.
	
	Now, consider the action of $G$ on its Cayley graph $\Gamma$: every $g \in G_P$ for $P \in V(Y)$ maps a vertex from $B(\Lambda)$ to another vertex in $B(\Lambda)$ (namely the elements of $\bigcup_{D \in P} \partial D$). Since the action of $G$ on $\GG$ is free, this means that $g$ has length at most $2\Lambda $. Likewise the image $g_D^{-1} g_{\ov D}$ of an edge $D \in P$ with $P \in V(Y)$~-- that is in particular $\partial D \sse B(\Lambda)$~-- maps $\partial \ov D \sse B(\Lambda + 1)$ into $ B(\Lambda)$ (since it maps the class $[\ov D]$ into $V(P)$). Therefore,
	\begin{align*}
	\abs{\phi(D)} &\leq 2\Lambda + 1 = 2\left(2(R+1)\cdot (\Theta + 1)\cdot\Theta \cdot \Xi + \Theta\right) + 1\\
	&\leq 4(R+1)\cdot (\Theta + 1)^2 \cdot \Xi.\qedhere
	\end{align*}
\end{proof}

\section{Stronger bounds for virtually free presentations}\label{sec:outer}


	Let us start with a virtually free group $G$ given as a free subgroup $F(X)$ of finite index together with a system of representatives $S$ of $F\bs G$. That means every group element can be written as $xs$ with $x \in F(X)$ and $s \in S$. Moreover, this normal form can be computed in linear time from an arbitrary word by successively applying ``commutation rules'' of letters from $S$ and $X \cup \ov X\cup  S$ to the word. This gives a virtually free presentation. For this special case, we can derive stronger bounds on the triangulation constant $k$ and other parameters. \ifArxiv These can be used later to show that in this case the graph of groups can be computed in \NP. \fi
 
%
Formally, a \emph{virtually free presentation} $\cvF$ for $G$ is given by the following data:
\begin{itemize}
	\item finite sets $X ,\ov{X}, S$, where $X \cup\ov{X}$ is a symmetric alphabet, $1 \in S$ and
	$(X \cup \ov{X}) \cap S= \emptyset$, 
	\item for all $y \in X \cup \ov X$, $r,t \in S \setminus \smallset{1}$, there are
	words $x_{r,y}, x_{r,t} \in (X \cup \ov X)^*,  s_{r,y} , s_{r,t} \in S$ such that
	\begin{align}
	ry &=_G x_{r,y}s_{r,y} & rt &=_G x_{r,t}s_{r,t}\label{eq:outer}
	\end{align}
\end{itemize} 
	fulfilling two properties:
	\begin{enumerate}
		\item for all $r \in S \setminus \smallset{1}$ there is some
		$ r' \in S \setminus \smallset{1}$ such that $s_{r',r}=1$ (\ie $G$ is a group),
		\item the equations (\ref{eq:outer}), oriented from left to right,
		together with the free reductions
		$x\ov{x}=1$ for  $x \in X \cup \ov{X}$
		form a {\em confluent} rewriting system  (for a definition, see \eg\ \cite{bo93springer,jan88eatcs}).		
	\end{enumerate}
Clearly such a presentation is terminating (noetherian), $F(X)$ is a subgroup of $G$,  
	 $G=F(X) \cdot S$, and
	$F(X) \cap S=\{1\}$ (hence $S$ is a system of right-representatives for $F(X)$).
	Note that properties (i) and (ii) can be checked in polynomial time.

Using this confluent rewriting system, every $g \in G$ can be uniquely written in its \emph{normal form} $g=xs$ where $x \in (X \cup \ov X)^*$ is a freely reduced word and $s \in S$. 
Given any word in $(X \cup \ov X \cup S)^*$, the normal form can be computed in linear time from left to right by applying the identities \prettyref{eq:outer} and reducing freely. This is the computation of a deterministic pushdown automaton for the word problem of $G$:

\begin{lemma}\label{lem:cfg_outer}
				Let $G$ be the group defined by a virtually free presentation $\cvF$. Then a deterministic pushdown automaton for $\WP{G}$ can be computed in polynomial time.
\end{lemma}

Notice that a finite extension of a free group is a special case of a virtually free presentation where $F(X)$ is a normal subgroup of $G$ (\ie $s_{r,y} = r$ for all $r \in S$, $y \in X \cup \ov X$). 
	
	We assume that $\cvF$ is written down in a naive way as input for algorithms: there is a table which for all $a \in X \cup \ov X \cup S $ and $r \in S$ contains a word $x_{r,a}$ and some $s_{r,a} \in S$. The size (number of letters) of this table is $ S\cdot(2X + S) \cdot \max\set{\abs{x_{r,a}}+1}{ a \in X \cup \ov X \cup S,\, r \in S}$.
	Up to logarithmic factors, this is the number of bits required to write down $\cvF$ this way.

	We can always add a disjoint copy of formal inverses $\ov S$ of $S$ representing $S^{-1}$ in $G$.
	Note that for $\ov s \in  \ov S$ this yields the rule $r \ov s = x_{r,\ov s}, s_{r,\ov s}$ for some $s_{r,\ov s} \in S$ and with $x_{r,\ov s} = x_{s_{r,\ov s},s}^{-1}$. In particular, $\abs{x_{r,\ov s}} \leq  \max\set{\abs{x_{r,a}}}{ a \in X \cup \ov X \cup S,\, r \in S}$. Therefore, 
	we define the size of $\cvF$ as $\Abs{\cvF} = S (2X + 2S) \cdot \max\set{\abs{x_{r,a}}+1}{ a \in X \cup \ov X \cup S,\, r \in S}$. 
	
	Whenever we talk about a group $G$ given as a virtually free presentation, we denote $\Sigma = X \cup \ov X \cup S \cup \ov S$. The respective Cayley graph $\Gamma = \Gamma_\Sig(G)$~-- as it is undirected~-- is defined with respect to this alphabet. So in particular, its degree is bounded by $\Abs{\cvF}$.
%

\begin{lemma}\label{lem:triangulation_k_tree}
	Let $T$ be a tree and let $v_0, \dots, v_n$ a sequence of vertices of $T$ such that $v_0 = v_n$ and $d(v_i,v_{i-1}) \leq k$ for some $k \in \N$. Then, the sequence $v_0, \dots, v_n$ is $k$-triangulable.
\end{lemma}

\begin{proof}
	Let $i$ be such that $d(v_i, v_0)$ is maximal. We will show that the sequence $u_0, \dots, u_{n-1}$ with $u_j = v_j$ for $j<i$ and $u_{j} = v_{j+1}$ for $j \geq i$ still satisfies $d(u_i,u_{i-1}) \leq k$. Then the lemma follows by induction.
	
	In particular, we only need to show that $d(v_{i-1},v_{i+1}) \leq k$. 
	Let $x$ be the last vertex which is shared by the geodesics from $v_0$ to $v_i$ and to $v_{i-1}$ and $y$ the last vertex shared by the geodesics from $v_0$ to $v_i$ and to $v_{i+1}$. Without loss of generality we may assume $d(v_0,x) \leq d(v_0,y)$ meaning that $d(x,v_i) = d(x,y) + d(y,v_i)$ and $d(x,v_{i+1}) = d(x,y) + d(y,v_{i+1})$.
	
	Since $d(v_0, v_i)$ is maximal, we have $d(y,v_{i+1}) \leq d(y,v_{i})$. Moreover, by the assumption of the lemma, we have
	$d(v_{i-1},v_i)  = d(v_{i-1},x) + d(x,v_i)  \leq k$.
	%
	Therefore, 
	\begin{align*}
	d(v_{i-1},v_{i+1}) &\leq  d(v_{i-1},x) + d(x,v_{i+1})\\
	&= d(v_{i-1},x) + d(x,y) + d(y,v_{i+1})\\
	&\leq d(v_{i-1},x) + d(x,y) + d(y,v_{i})\\
	&=(v_{i-1},x) + d(x,v_i) = d(v_{i-1},v_i) \leq k\qedhere
	\end{align*}
\end{proof}

	\begin{lemma}\label{lem:triangulation_k_outer}
		Let $G$ be the group defined by a virtually free presentation $\cvF$ 
and let $\Gamma$ be its Cayley graph. Then $\Gamma$ is $k$-triangulable for 
$k= 2\Abs{\cvF} + 2$.
	\end{lemma}

\begin{proof}
	Consider a closed path $v_0, \dots, v_n$ (\ie for all $i$ there is some $a \in \Sigma$ with $v_i = v_{i-1} \cdot a$ and $v_n = v_0$) and let $x_i s_i$ for $x_i\in F(X)$, $s_i \in S$ be the normal form of $v_i$ for $i=0,\dots, n$. We consider the sequence $x_0, \dots, x_n$ in the Cayley graph of $F(X)$. Since $x_is_i = x_{i-1}s_{i-1}a$ for some $a \in \Sigma$, we have
	\begin{align*}
	x_i 	&\,=\, x_{i-1}s_{i-1}as_i^{-1}
	\,=\, x_{i-1} x_{s_{i-1},a}  s_{s_{i-1},a} \ov s_i
	\,=\, x_{i-1} x_{s_{i-1},a}  x_{s_{s_{i-1},a}, \ov s_i} s_{s_{s_{i-1},a}, \ov s_i}.
	\end{align*}
	
	Therefore, we have $d(x_i,x_{i+1}) \leq 2\max\set{\abs{x_{r,a}}+1}{ a \in X \cup \ov X \cup S,\, r \in S} \leq 2\Abs{\cvF}$	
	and so the sequence $x_0, \dots, x_n$ can be $2\Abs{\cvF}$-triangulated by \prettyref{lem:triangulation_k_tree}. Now such a  $2\Abs{\cvF}$-triangulation is a  $2\Abs{\cvF} + 2$-triangulation of the original path $v_0, \dots, v_n$.
\end{proof}

	\begin{lemma}\label{lem:finite_subgroups_outer_action}
		Let $G$ be the group defined by a virtually free presentation $\cvF$. Then for every finite subgroup $H \leq G$, we have $\abs{H} \leq\abs{S} $. Hence, in particular, $\abs{H} \leq\Abs{\cvF}$.
	\end{lemma}
	\begin{proof}
		The group $G$ acts on the quotient $F(X) \bs G$ from the right with stabilizers being conjugates of $F(X)$. Since $H$ is finite, it has trivial intersection with any conjugate of $F(X)$ and so it acts freely on $F(X) \bs G$. Thus, there is an injective map  $H \to F(X) \bs G$. Since $S$ is a system of representatives of $F(X) \bs G$, it follows $\abs{H} \leq \abs{S} \leq \Abs{\cvF}$.
	\end{proof}
	
	\begin{lemma}\label{lem:gog_size_outer}
		Let $G$ be the group defined by a virtually free presentation $\cvF$. Then the number of edges of a reduced graph of groups for $G$ with finite vertex groups is at most $\Abs{\cvF}$.	
	\end{lemma}

\begin{proof}
	The proof is analogous to \cite[Thm.\ 1.4]{sen96dimacs}: by \cite[Thm.\ 2]{Linnell83}, we have 
	\[2(X + S) \geq 1 + \sum_{e \in E}\frac{1}{\abs{G_e}}\]
	where $E$ are the undirected edges of the graph of groups and $G_e$ the respective edge groups (note that here every undirected edge $\oneset{e,\ov e}$ is counted only once). Since all edge groups are finite, we obtain
	\begin{align*}
	\abs{E} &\leq 2(X + S) \cdot \max_{e\in E}\abs{G_e}\\
	&\leq 2(X + S) \cdot \abs{S} \tag{by \prettyref{lem:finite_subgroups_outer_action}}\\
	&\leq \Abs{\cvF}.\qedhere
	\end{align*}
\end{proof}

	\begin{lemma}\label{lem:cut_size_outer}
		Let $G$ be the group defined by a virtually free presentation $\cvF$ and let $\Gamma$ be its Cayley graph. Then every minimal cut in $\Gamma$ is a $K$-cut for $K = \Abs{\cvF}^2$.
	\end{lemma}

\begin{proof}
	Consider the set of normal forms: every group element $g\in G$ can be written as $x_gs_g$ for a unique freely reduced word $x_g \in (X \cup \ov X)^*$ and $s_g \in S$.  Given some freely reduced $x \in (X \cup \ov X)^*$, it defines a cut $C_x  = \set{ys}{s \in S,\, x \text{ is a prefix of } y}$ (clearly $C_x$ and $\Comp C_x$ are connected~-- below we prove that $\vecdel C_x$ is finite). Notice that every bi-infinite simple path in $\Gamma$ which can be split by any cut also can be split by some cut of the form $C_x$.
	Hence, it remains to bound $\vecdel C_x$.

	Consider some 
	$(yr,\ov a) \in \vecdel C_x $, \ie $yr \in C_x$ and $yr\ov a =_G zs \in \Comp C_x$ for $a \in X\cup \ov X \cup S \cup \ov S$, $y,z \in (X \cup \ov X)^*$ and $r,s \in S$.  Thus, we have $z x_{s,a} =_{F(X)} y$. Since $z$, $x_{s,a}$ and $y$ are freely reduced, there are freely reduced $v, z',x'_{s,a} \in (X \cup \ov X)^*$ such that $z=z'v$ and $ x_{s,a} = \ov v  x'_{s,a}$ and $y = z'x'_{s,a}$ as words.
	Since $x$ is a prefix of $y$ but not of $z$, it follows that $x$ is a prefix of $z'x'_{s,a}$ but not of $z'$. 
	
	Hence, if we fix $x'_{s,a}$, then there are at most $\abs{x'_{s,a}}$ many possibilities for $z'$ (namely the prefixes of $z'$ of length $\abs{z'} - i$ for $i = 1, \dots, \abs{x'_{s,a}}$). 
	Moreover, if we fix only $a$ and $s$, then $x'_{s,a}$ can be any suffix of $x_{s,a}$. Thus, there are at most $\sum_{i=1}^{\abs{x_{s,a}}}i \leq \abs{x_{s,a}}^2$ many possibilities for $z'$ if $a$ and $s$ are known.  Knowing $x'_{s,a}$, $x_{s,a}$, and $z'$ determines also $z$ and $y$ completely.
	Thus, summing up over all different $a$ and $s$, we obtain
	\begin{align*}
	\vert \vecdel C_x \vert &\leq \sum_{s \in S} \sum_{a \in  X\cup \ov X \cup S \cup \ov S}\abs{x_{s,a}}^2
	\leq \abs{S} \cdot\abs{X\cup \ov X \cup S \cup \ov S} \cdot \!\!\!\!\max_{s \in S \atop a \in  X\cup \ov X \cup S \cup \ov S} \!\abs{x_{s,a}}^2
	\leq \Abs{\cvF}^2.\qedhere
	\end{align*}
\end{proof}

	\section{Main results: computing graphs of groups}\label{sec:main}
	
	\begin{lemma}\label{lem:rational_membership}
		The uniform rational subset membership problem for virtually free groups given as virtually free presentation or as context-free grammar for the word problem can be decided in polynomial time. More precisely, the input is given as
		\begin{itemize}
			\item either a virtually free presentation $\cvF$ or a context-free grammar $\Gram = (V,\Sigma, P,S)$ for the word problem of a group $G$,
			\item a rational subset of $G$ given as non-deterministic finite automaton or regular expression over $\Sigma$ (in the case of a virtually free presentation $\Sigma$ is defined as in \prettyref{sec:outer}),
			\item a word $w \in \Sigma^*$.
		\end{itemize}
		The question is whether $w$ is contained in the rational subset of $G$.
	\end{lemma}
\begin{proof}
	For this proof we will distinguish between group elements and words. Let $p: \Sigma^* \to G$ denote the canonical projection.
	Since a regular expression can be transformed into a finite automaton with only linear overhead (see \eg \cite{HU}), we can assume that the input is given as a non-deterministic finite automaton accepting the regular language $L\sse \Sigma^*$. The question now is whether $p(w) \in p(L)$.
	
	From this automaton, we construct a new automaton for $ L' = \ov w L$
	by adding $\abs{w}$ new states. Clearly, 
	$p(w) \in p(L)$ if, and only if, $1\in p(L')$
	or with other words $p^{-1}(1) \cap L' \neq \emptyset$. 
	The latter can be tested by
	computing a pushdown-automaton for $p^{-1}(1)$~-- either with
	the standard construction from a context-free grammar \cite{HU} or
	from the virtually free presentation as in \cite{ms83}. Both
	constructions can be done in polynomial time. From this push-down
	automaton we can easily construct a new pushdown-automaton for
	$p^{-1}(1) \cap L'$, 
	which then can be checked for emptiness (for
	both constructions see \cite{HU}).
\end{proof}

\begin{proposition}\label{prop:verify_gog}
		The following problem is in \P:
		Given a virtually free group $G$ either as virtually free presentation $\cvF$ or as context-free grammar $\Gram$ for its word problem and a graph of groups $\cG$ over the graph $Y$ (with vertex groups as multiplication tables, \ie for all $g,h \in G_P$ the product $gh$ is written down explicitly) together with a homomorphism $\phi: \Delta^* \to \Sigma^*$ (where $\Delta = \bigcup_{P\in V(Y)} G_P \cup E(Y)$ and $\Sigma$ is the generating set for $G$ defined by $\cvF$ (resp.\ $\Gram$)), decide whether $\phi$ induces an isomorphism $\pi_1(\cG,T) \to G$.
\end{proposition}
\begin{proof}

	We verify that $\phi$ induces a homomorphism $\tilde \phi: \pi_1(\cG,T) \to G$ and that $\tilde \phi$ is injective and surjective.

 Testing that $\phi$ really induces a \emph{homomorphism} reduces to the word problem
	for the group $G$, which can be solved in polynomial time: for every relation $a_1 \cdots a_m=1$ of $\pi_1(\cG,T)$ test whether $\phi(a_1) \cdots \phi(a_m) = 1$ in $G$.
 	Testing that $\tilde \phi$ is \emph{surjective} reduces to polynomially many
	membership-problems for rational subsets of $G$: for all $a \in \Sigma$ test whether $a$ is contained in the rational subset $\set{\tilde\phi(g)}{g\in \Delta}^*$. By \prettyref{lem:rational_membership} this can be done in polynomial time.

		It remains to test whether $\tilde\phi$ is \emph{injective}. Let $\pi: \Delta^* \to F(\cG)$ and $F(\cG) \to \pi_1(\cG,T)$ denote the canonical projections (note that $\psi$ induces an isomorphism $\pi_1(\cG,P) \tto \pi_1(\cG,T)$). Let $\cR \sse \Delta^*$ denote the set of reduced words. With slight abuse of notation we use $\pi_1(\cG,P)$ also to denote the set of words $g_0y_1\cdots g_{n-1}y_ng_{n}\in \Delta^*$ where $y_1\cdots y_n$ is a closed path based at $P$ and $g_i\in G_{s(y_{i+1})}$  for $0\leq i < n$ and $ g_{n} \in G_{P}$.
	Testing that $\tilde\phi$ is injective amounts to test whether the language 
	\[L = \bigl(\pi^{-1}(\psi^{-1}(\tilde\phi^{-1}(1))) \cap \pi_1(\cG,P)	 \cap \cR \bigr) \setminus \{1\} \sse \Delta^*\]
	is empty because $1$ is the only reduced word in $\pi_1(\cG,P)$ representing the identity, by \prettyref{lem:reduced_word}. 

	Notice that $\pi^{-1}(\psi^{-1}(\tilde\phi^{-1}(1))) = \phi^{-1}(\WP{G})$. Since $\WP{G}$ is context-free (for virtually free presentations, see \prettyref{lem:cfg_outer}) and since context-free languages are closed under inverse homomorphism, $\varphi^{-1}(\WP{G})$ is a context-free language~-- and a pushdown automaton for it can be computed in polynomial time from the pushdown automaton for $\WP{G}$ (see \cite{HU}). 

Both $\pi_1(\cG,P)$ and $\cR$ 
are regular languages with finite automata of size polynomial in the size of 
the graph of groups (and they can be computed in polynomial time):  
	For  $ \pi_1(\cG,P) $ take as deterministic finite automaton the graph $Y$ plus one additional fail state. In state $P \in V(Y)$ one can read an element of $G_P$ and stay in $P$ or read an outgoing edge and go to its terminal vertex~-- all other transitions go to the fail state.
For $\cR$, we observe that it is equal to $\Delta^* \setminus \Delta^*\wt{R}\Delta^*$ where $\wt{R}$ is the set 
of {\em forbidden} factors:
\[\wt{R} := \set{ gh}{ \exists P \in V(Y),g,h \in G_P \setminus \{1\} } \cup \oneset{\ov{y} a ^y y \mid y \in E(Y),a \in G_y }.\]
$\cR$ is thus recognized by a deterministic finite automaton $\mathbb{A}_\cR$ that, after reading a word $u$, 
memorizes in its state the longest suffix of $u$ which is a prefix 
of $\wt{R}$ (plus a fail state if it has encountered a full member of $\wt{R}$).
Since the  set of proper prefixes of $\wt{R}$ has linear size, $\mathbb{A}_\cR$ has a
quadratic number of transitions. 

Thus, $L$ is a context-free language and we obtain a pushdown automaton for $L$, 
which can be tested for emptiness  in polynomial time 
(for both constructions see \eg\ \cite{HU}). 
\end{proof}

\begin{theorem}\label{thm:compute_gog_cfg}
	The following problem is in $\NTIME(2^{2^{\Oh(N)}})$:\\	
	\noindent\emph{Input:} a context-free grammar $\Gram = (V,\Sigma,P,S)$ in Chomsky normal form with $\Abs{\Gram} \leq N$ which generates the word problem of a group $G$,\\	
	\noindent\emph{Compute} a graph of groups with finite vertex groups and fundamental group $G$.
\end{theorem}
	Note that if $\Gram$ is not in Chomsky normal form, it can be transformed into Chomsky normal form in quadratic time. In this case the graph of groups can be computed in $\NTIME(2^{2^{\Oh(N^2)}})$.

	\begin{theorem}\label{thm:compute_gog_NP_outer}
		The following problem is in $\NP$:\\		
		\noindent\emph{Input:} a group $G$ as a virtually free presentation,\\	
		\noindent\emph{Compute} a graph of groups with finite vertex groups and fundamental group $G$.
	\end{theorem}

		\begin{table}[thb]
	\small
	\caption{\small Summary of appearing constants. The third column shows a bound in terms of the size of a context-free grammar $\Gram$ for the word problem (due to \prettyref{lem:triangulation_k}, \prettyref{lem:diameter_mincuts}, \prettyref{prop:subgroupbound}, and \prettyref{thm:1.4_dimacs96}), the fourth column shows a bound in terms of the size of a virtually free presentation $\cvF$ (due to \prettyref{lem:triangulation_k_outer}, \prettyref{lem:cut_size_outer}, \prettyref{lem:lemma6_icalp93}, \prettyref{lem:finite_subgroups_outer_action}, and \prettyref{lem:cut_size_outer}).\smallskip}\label{tab:constants}
	\begin{tabularx}{\textwidth}{l l l l}
		$N $ & size of the input &  $\Abs{\Gram}$ & $\Abs{\cvF}$  \\
		$d     $ & degree of $\Gamma$	&  $N $ & $N$\\
		$k     $ & triangulation constant	&   $ 2^{N+2} $ & $ 2N+2$\\
		$K     $ & maximal weight of a minimal cut	&  $ d^{3k+3}  $ 	&  $ N^2 $\\
		$R = \frac{3kK}{2}    $& maximal diameter of the boundary of a minimal cut	&  $ \frac{3k}{2}d^{3k+3} \quad $ 	&  $ 3(N+1)N^2 $\\	
		$\Xi $ & maximum cardinality of a finite subgroup & $ d^{12k + 10}$  & $N$   \\
		$\Theta $ & maximum number of edges in a reduced graph of groups & $ d^{12k + 11}$  & $N$   \\
	\end{tabularx}
\end{table}

	\begin{proof}[Proof of \prettyref{thm:compute_gog_NP_outer} and \prettyref{thm:compute_gog_cfg}]
	We give one proof for both theorems. The difference is only in the form the input is given. Let $R$, $\Theta$ and $\Xi$ are defined as in \prettyref{lem:gog_bound} and before and let $N$ denote the size of the input (\ie either $N=\Abs{\cvF}$ or $N=\Abs{\Gram}$).
	Consider the case of \prettyref{thm:compute_gog_cfg}~-- \ie the input is given as context-free grammar $\Gram = (V,\Sigma, P,S)$. By \prettyref{tab:constants}, we have $k \in \Oh(2^N)$ where $k$ is the triangulation constant of $\Gamma$ and $\Xi, \Theta, R \in d^{\Oh(k)} \sse 2^{2^{\Oh(N)}}$. 
	Now consider the case of \prettyref{thm:compute_gog_NP_outer}~-- \ie the input is given as a virtually free presentation $\cvF$. By \prettyref{tab:constants}, $\Xi, \Theta, R \in  \Oh(\Abs{\cvF}^\ell)$ for some $\ell \in \N$.
	Thus, by \prettyref{lem:gog_bound} there is a graph of groups $\cG$ and an isomorphism $\phi:\pi_1(\cG,T) \to G$ such that 
	\begin{enumerate}
		\item $\Card{V(Y)} \in 2^{2^{\Oh(N)}}$\hfill (resp.\ $\Card{V(Y)} \in \Oh(\Abs{\cvF}^\ell)$)
		\item $\Card{G_P} \in 2^{2^{\Oh(N)}} $ for all $P \in V(Y)$ \hfill (resp.\ $\Card{G_P} \in \Oh(\Abs{\cvF}^\ell)$)
		\item  $|\phi(a)| \in 2^{2^{\Oh(N)}}$ for every $a \in \bigcup_{P\in V(Y)} G_P\cup E(Y)$ \hfill(resp.\ $|\phi(a)| \in \Oh(\Abs{\cvF}^\ell)$). 
	\end{enumerate}
	Let $\Delta = \bigcup_{P\in V(Y)} G_P \cup E(Y)$ denote the alphabet from \prettyref{sec:bass_serre}. 
	Our algorithm consists of two steps: first we guess the graph of groups $\cG$ and a map $\phi:\Delta \to G$ within the size bounds of \prettyref{lem:gog_bound} and second we verify that $\phi$ is indeed an isomorphism $\pi_1(\cG,T) \to G$ using \prettyref{prop:verify_gog}. \prettyref{lem:gog_bound} assures that there is a valid guess in the first step and by (i)--(iii) guessing can be done within the time bounds of \prettyref{thm:compute_gog_cfg} (resp.\ \prettyref{thm:compute_gog_NP_outer}). Notice that $2^{p(2^{\Oh(N)})}= 2^{2^{\Oh(N)}}$ for any polynomial $p$, so \prettyref{prop:verify_gog} yields the desired time bound also in the case of a grammar as input.
\end{proof}

\section{Slide moves and the isomorphism problem}\label{sec:iso}

Given two groups $G_1$ and $G_2$ one can calculate the respective graph of groups and then check with Krstic's algorithm by  (\cite{krstic89}) whether their fundamental groups are isomorphic. A closer analysis shows that this algorithm runs in polynomial space. As the description is quite involved, we follow a different approach based Forester's theory of deformation spaces \cite{Forester02,ClayF09}.

\ifArxiv
Let $\cG_1$, $\cG_2$ be the two graph of groups obtained from the presentations of $G_1$, $G_2$ (with the vertex groups given as multiplication tables). By construction, we have $\pi_1(\cG_1,P_1) \cong G_1$ and $\pi_1(\cG_2,P_2) \cong G_2$. We will test in \PSPACE (more precisely linear space) whether $\pi_1(\cG_1,P_1) \cong \pi_1(\cG_2,P_2) $.
\fi

\ifArxiv
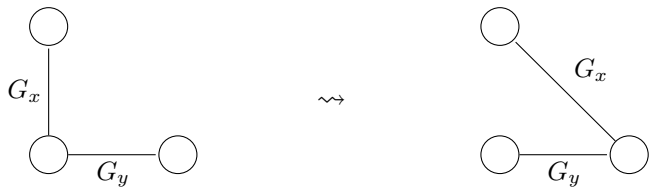
\begin{figure}
	\begin{center}\begin{minipage}{.3\textwidth}
			\begin{center}
				\small
				\begin{tikzpicture}[ >=latex,shorten > =1pt,auto,initial text={}, every state/.style={minimum size=5mm}, node distance=1.7cm]
				\node[state] (1) {};
				\node[state,below of=1] (2) {};
				\node[state,right of=2] (3) {};

				\draw[-] (2) to node[left=-1pt] {$G_x$} (1);
				\draw[-] (2) to node[below=-1pt] {$G_y$} (3);
				\end{tikzpicture}
			\end{center}			
		\end{minipage}
		\begin{minipage}{.1\textwidth}
			\begin{center}
				$\leadsto$
			\end{center}	
		\end{minipage}
		\begin{minipage}{.3\textwidth}
			\begin{center}
				\small
				\begin{tikzpicture}[ >=latex,shorten > =1pt,auto,initial text={}, every state/.style={minimum size=5mm}, node distance=1.7cm]
				\node[state] (1) {};
				\node[state,below of=1] (2) {};
				\node[state,right of=2] (3) {};

				\draw[-] (3) to node[above right ] {$G_x$} (1);
				\draw[-] (2) to node[below=-1pt] {$G_y$} (3);
				\end{tikzpicture}
			\end{center}
		\end{minipage}
	\end{center}
	\caption{The slide move is possible if there is some $g \in G_P$ such that $g^{-1} G_x^x g \leq G_y^y$.}\label{fig:slide_move}
\end{figure}
\fi

Let $\cG$ be a graph of groups over $Y$. A \emph{slide move} is the following operation on $\cG$: let $G_P$ be a vertex group and $G_x$, $G_y$ edge groups with $s(x) = s(y) = P$. If  $G_x^x$ (the image of $G_x$ in $G_P$) can be conjugated by an element of $G_P$ into $G_y^y$ \ie there is some $g \in G_P$ such that $g^{-1} G_x^x g \leq G_y^y$, then $x$ can be slid along $y$ to $Q = t(y)$, \ie $s(x)$ is changed to $Q$\ifArxiv, see \prettyref{fig:slide_move}\fi. The new inclusion of $G_x \to G_Q$ is then given by $ \iota_{\ov y} \circ \iota_y^{-1 }\circ c_g\circ\iota_x$ where $\iota_x$ is the inclusion $G_x \to G_P$ (likewise for $\iota_{\ov y}, \iota_y$) and $c_g$ is the conjugation with $g$ (\ie $h \mapsto g^{-1} h g$). 
A slide move induces an isomorphism $\phi$ of the fundamental groups of the two graph of groups by $\phi(h) = h$ for $h \in G_R$, $R\in V(Y)$, and $\phi(z)= z$ for $z \in E(Y) \setminus \oneset{x, \ov x}$ and $\phi(x) = gyx$.

\ifArxiv Note that slide moves originally were defined on the trees instead of graph of groups. On trees there is no need for a conjugation before sliding an edge~-- the conjugation corresponds simply to choosing a different representative for an edge. In \cite{Bass93} it is shown that two graph of groups have the same Bass-Serre tree if and only if they are equal up to conjugation of the inclusions of the edge groups into the vertex groups (this is also not very hard to see when following the construction of a graph of groups from an action on a tree in \prettyref{sec:GX}). \fi

The following result is an immediate consequence of \cite[Thm.~1.1]{Forester02} and \cite[Cor.~3.5]{ClayF09} (resp.\ \cite[Thm.~7.2]{GuirardelL07}). 
\ifArxiv
	Since we are not aware of an explicit reference, we present the proof.
 \else
	Since we are not aware of an explicit reference, we present the details in \cite{SenizerguesW18arxiv}.
\fi
\begin{proposition}\label{prop:forester}
	Let $\cG_1$ and $\cG_2$ be reduced finite graph of groups with finite vertex group. Then $\pi_1(\cG_1,P_1) \cong \pi_1(\cG_2,P_2) $ if and only if $\cG_1$ can be transformed into $\cG_2$ by a sequence of slide moves.
\end{proposition}

\begin{proof}
	As described above, slide moves induce isomorphisms on the fundamental groups. Thus, let $G=\pi_1(\cG_1,P_1) \cong \pi_1(\cG_2,P_2)$. This gives us two different actions of $G$ on the respective Bass-Serre trees. Both actions have the same elliptic subgroups, namely all finite subgroups.
	By \cite[Thm.~1.1]{Forester02}, $\cG_1$ can be transformed into $\cG_2$ by a sequence of elementary deformations. Since the corresponding deformation space is non-ascending (meaning that there are no self-loops in the graph of groups with one inclusion of the edge group being surjective but the other not~-- this clearly cannot happen since all vertex groups are finite), by \cite[Cor.~3.5]{ClayF09} or \cite[Thm.~7.2]{GuirardelL07}, $\cG_1$ actually can be transformed into $\cG_2$ by a sequence of slide moves.
\end{proof}

Clearly, any sequence of slide moves can be performed in linear space. By guessing a sequence of slide moves transforming $\cG_1$ into $\cG_2$, we obtain the following corollary.

\begin{corollary}
	Given two graph of groups $\cG_1$ and $\cG_2$ where all vertex groups are given as full multiplication tables, it can be checked in $\NSPACE(\Oh(n))$ whether $\pi_1(\cG_1,P_1) \cong \pi_1(\cG_2,P_2) $. 
\end{corollary}

In combination with \prettyref{thm:compute_gog_cfg} (and Savitch's theorem) and \prettyref{thm:compute_gog_NP_outer} this gives an algorithm to solve the isomorphism problem for virtually free groups:
\begin{theorem}\label{thm:iso_NEXP2_cf}
	The isomorphism problem for context-free groups is in $\DSPACE(2^{2^{\Oh(N)}})$. More precisely, the input is given as two context-free grammars of size at most $N$ which are guaranteed to generate word problems of groups. 
\end{theorem}

\begin{theorem}\label{thm:iso_PSPACE_outer}
	The isomorphism problem for virtually free groups given as a virtually free presentation is in \PSPACE.
\end{theorem}

\vspace{-3mm}
\section{Conclusion and open questions}\label{sec:conclusion}
We have shown that the isomorphism problem for virtually free groups is in \PSPACE (resp.\ $\DSPACE(2^{2^{\Oh(N)}})$) depending on the type of input~-- thus, improving the previous bound (primitive recursive) significantly. The following questions remain open:				
	\begin{itemize}
		\item What is the complexity of the isomorphism problem for virtually free groups given as an arbitrary presentation?
		\item Is the doubly exponential bound $n^{12\cdot 2^n + 10}$ on the size of finite subgroups tight or is there a bound $2^{p(n)}$ for some polynomial $p$? This is closely related to another question: 
		\item What is the minimal size of a context-free grammar of the word problem of a finite group? Can it be $\log \log (n)$ where $n$ is the size of the group?
		\item Is there a polynomial bound on the number of slide moves necessary to transform two graphs of groups with isomorphic fundamental groups into each? This would lead to an \NP algorithm for the isomorphism problem with virtually free presentations as input. We conjecture, however, that this is not true.
	\end{itemize}	

\vspace{-3mm}
\paragraph*{Acknowledgements.}
G.S. thanks the FMI for hosting him from October to the end of the year 2017.
Both authors acknowledge the financial support by the DFG project DI 435/7-1 ``Algorithmic problems in group theory'' for this work.




\newcommand{\Ju}{Ju}\newcommand{\Ph}{Ph}\newcommand{\Th}{Th}\newcommand{\Ch}{Ch}\newcommand{\Yu}{Yu}\newcommand{\Zh}{Zh}\newcommand{\St}{St}\newcommand{\curlybraces}[1]{\{#1\}}

%

\end{document}